\documentclass[12pt, reqno]{amsart}

\usepackage{amsthm}
\usepackage{amssymb, amsmath}

\usepackage{amscd}   % for commutative diagrams
\usepackage[all]{xy} % for complicated commutative diagrams
\usepackage{colonequals}
\usepackage{tikz-cd}
\usepackage[shortlabels]{enumitem}

\usepackage[backref, colorlinks=true, linkcolor=blue, citecolor=blue]{hyperref}
\usepackage[alphabetic,backrefs,lite]{amsrefs}

\usepackage[]{xcolor}

\usepackage[top=1in, bottom =1in, left=1in, right = 1in]{geometry}

\usepackage{comment}
\DeclareFontEncoding{OT2}{}{} % to enable usage of cyrillic fonts

\hypersetup{
pdflang={en-US},
pdftitle={},
pdfauthor={Sarah Frei, Lena Ji, Soumya Sankar, Bianca Viray, and Isabel Vogt},
pdfsubject={Algebraic Geometry},
pdfkeywords={conic bundles, rationality, rationality obstructions, prym varieties, intermediate jacobians}
}

\makeatletter
\renewcommand*\env@matrix[1][*\c@MaxMatrixCols c]{%
  \hskip -\arraycolsep
  \let\@ifnextchar\new@ifnextchar
  \array{#1}}
\makeatother

\newtheorem{lemma}{Lemma}[section]
\newtheorem{theorem}[lemma]{Theorem}
\newtheorem{proposition}[lemma]{Proposition}
\newtheorem{prop}[lemma]{Proposition}
\newtheorem{cor}[lemma]{Corollary}

\newtheorem{claim*}{Claim}
\newtheorem{thm}[lemma]{Theorem}
\newtheorem*{thm*}{Theorem}
\newtheorem{defn}[lemma]{Definition}

\theoremstyle{remark}
\newtheorem{remark}[lemma]{Remark}
\newtheorem{remarks}[lemma]{Remarks}

\newcommand{\G}{{\mathbb G}}

\newcommand{\PP}{{\mathbb P}}
\newcommand{\C}{{\mathbb C}}
\newcommand{\F}{{\mathbb F}}
\newcommand{\Q}{{\mathbb Q}}
\newcommand{\RR}{{\mathbb R}}
\newcommand{\Z}{{\mathbb Z}}

\newcommand{\kbar}{{\overline{k}}}

\newcommand{\kk}{{\mathbf k}}

\newcommand{\calF}{{\mathcal F}}

\newcommand{\OO}{{\mathcal O}}

\usepackage[mathscr]{euscript}

\DeclareMathOperator{\HH}{H}
\DeclareMathOperator{\hh}{h}

\DeclareMathOperator{\rk}{rk}

\DeclareMathOperator{\im}{im}

\DeclareMathOperator{\Aut}{Aut}
\DeclareMathOperator{\Gal}{Gal}

\DeclareMathOperator{\Br}{Br}

\DeclareMathOperator{\Pic}{Pic}
\DeclareMathOperator{\bPic}{{\bf Pic}}

\DeclareMathOperator{\Spec}{Spec}

\DeclareMathOperator{\Frac}{Frac}

\DeclareMathOperator{\PGL}{PGL}

\DeclareMathOperator{\et}{et}

\DeclareMathOperator{\rank}{rank}

\DeclareMathOperator{\Diag}{Diag}

\DeclareMathOperator{\disc}{disc}

\DeclareMathOperator{\Bl}{Bl}

\newcommand{\Deltatilde}{\tilde{\Delta}}

\newcommand{\alphatildem}[1]{\tilde{\alpha}^{(#1)}}
\newcommand{\Ptilde}{\tilde{P}}
\newcommand{\Ptildem}[1]{\tilde{P}^{(#1)}}

\newcommand{\Pm}[1]{{P}^{(#1)}}

\newcommand{\piDelta}{\varpi}
\newcommand{\piX}{\pi}

\newcommand{\piQuadricSurface}{\piX_1}

\newcommand{\Span}{\operatorname{Span}}

\numberwithin{equation}{section}
\numberwithin{table}{section}

\newcommand{\defi}[1]{\textsf{#1}} % for defined terms

\title[Conic bundles differing by a constant Brauer class and rationality]{Conic bundle threefolds differing by a constant Brauer class and connections to rationality}

\author{Sarah Frei}
\address{Department of Mathematics, Rice University, 6100 S.\ Main St., Houston, TX 77005, USA}
\email{sarah.frei@rice.edu}
\urladdr{https://math.rice.edu/\~{}sf31}

\author{Lena Ji}
\address{Department of Mathematics, University of Illinois Urbana-Champaign, 214 Harker Hall, 1305 W. Green Street, Urbana, IL 61801}
\email{lenaji.math@gmail.com}
\urladdr{https://lji.web.illinois.edu/}

\author{Soumya Sankar}
\address{Mathematical Institute, Utrecht University, Hans Freudenthal building, Budapest 6, 3584 CD Utrecht, The Netherlands}
\email{s.sankar@uu.nl}
\urladdr{https://sites.google.com/site/soumya3sankar/}

\author{Bianca Viray}
\address{University of Washington, Department of Mathematics, Box 354350, Seattle, WA 98195, USA}
\email{bviray@uw.edu}
\urladdr{http://math.washington.edu/\~{}bviray}

\author{Isabel Vogt}
\address{Brown University, Department of Mathematics, Box 1917, 151 Thayer Street, Providence, RI 02912, USA}
\email{ivogt.math@gmail.com}
\urladdr{https://www.math.brown.edu/ivogt/}


\begin{document}

\begin{abstract}
   A double cover \(Y\) of \(\PP^1 \times \PP^2\) ramified over a general \((2,2)\)-divisor will have the structure of a geometrically standard conic bundle ramified over a smooth plane quartic \(\Delta \subset \PP^2\) via the second projection.  These threefolds are rational over algebraically closed fields; however, over nonclosed fields, including \(\RR\), their rationality is an open problem.
   In this paper, we characterize rationality over \(\RR\) when \(\Delta(\RR)\) has at least two connected components (extending work of M. Ji and the second author) and over local fields when all odd degree fibers of the first projection have nonsquare discriminant.

   We obtain these applications by proving general results comparing the conic bundle structure on \(Y\) with the conic bundle structure on a well-chosen intersection of two quadrics. The difference between these two conic bundles is encoded by a constant Brauer class, and we prove that this class encodes the obstruction to the existence of a section of the first projection \(Y\to\PP^1\).
\end{abstract}
\maketitle
%%%%%%%%%%%%%%%%%%%%%%%%%%%%%%%%%%%%%%%%%%%%%%%%%%%%%%%%%%%%%%%%%%%%%%%%%%%%%%%%%%%%%%%%%%%%%%%%
%%%%%%%%%%%%%%%%%%%%%%%%%%%%%%%%%%%%%%%%%%%%%%%%%%%%%%%%%%%%%%%%%%%%%%%%%%%%%%%%%%%%%%%%%%%%%%%%
\section{Introduction}
%%%%%%%%%%%%%%%%%%%%%%%%%%%%%%%%%%%%%%%%%%%%%%%%%%%%%%%%%%%%%%%%%%%%%%%%%%%%%%%%%%%%%%%%%%%%%%%%
%%%%%%%%%%%%%%%%%%%%%%%%%%%%%%%%%%%%%%%%%%%%%%%%%%%%%%%%%%%%%%%%%%%%%%%%%%%%%%%%%%%%%%%%%%%%%%%%

%%%%%%%%%%%%%%%%%%%%%%%%%%%%%%%%%%%%%%%%%%%%%%%%%%%%%%%%%%%%%%%%%%%%%%%%%%%%%%%

For threefolds over algebraically closed fields, much is known about the rationality problem (i.e., the property of being birational to projective space), see e.g.~\cite{IskovskikhProkhorov}. However, the classical rationality obstructions cannot always detect irrationality over nonclosed fields. In this direction,
Hassett--Tschinkel and Benoist--Wittenberg introduced a rationality obstruction over nonclosed fields \(k\) that refines the classical Clemens--Griffiths intermediate Jacobian obstruction~\cite{clemensgriffiths} by considering certain torsors under the intermediate Jacobian~\cites{bw-cg, HT-intersection-quadrics, bw-ij, HT-cycle}.  
Subsequent works show that this refined \defi{intermediate Jacobian torsor (IJT) obstruction} is strong enough to characterize rationality for many classes of Fano threefolds with \(k\)-points \cites{HT-intersection-quadrics, bw-ij, KP-Fano-3folds-rank1}.

In an earlier paper \cite{FJSVV}, the authors studied the IJT obstruction for geometrically standard conic bundle threefolds over \(\PP^2\) with singular fibers over a smooth plane quartic \(\Delta\), and proved that although the IJT obstruction characterizes rationality for these degree \(4\) conic bundles over fields \(k\) with \(\Br k[2] = 0\), the IJT obstruction fails to characterize rationality over any subfield of \(\RR\).  The root cause of this discrepancy is that the intermediate Jacobian torsors of a conic bundle \(\pi\colon X \to \PP^2\) depend only on the discriminant cover \(\Deltatilde\to \Delta\) of \(\pi\), whereas the isomorphism class of the generic fiber \(X_{\eta}\) depends on the discriminant cover \emph{and} the constant \(2\)-torsion Brauer class given by a smooth fiber at any point \(w\in(\PP^2\smallsetminus\Delta)(k)\).

Thus, one may ask whether the vanishing of the IJT obstruction for a degree \(4\) conic bundle \(X\to \PP^2\) implies the rationality of \emph{some} degree \(4\) conic bundle \(X'\to \PP^2\) with the same discriminant cover \(\Deltatilde/\Delta\) as \(X\).  
We prove that this indeed occurs when the
 \defi{Prym curve}\footnote{The intermediate Jacobian of \(X\) is the Prym variety of \(\Deltatilde/\Delta\), which, in the case \(\deg \Delta = 4\), is isomorphic to the Jacobian of an explicit genus \(2\) curve \(\Gamma_{\Deltatilde/\Delta}\) with equation given in \eqref{eq:Gamma} \cite{FJSVV}*{Definition 4.6}.} \(\Gamma_{\Deltatilde/\Delta}\) has a \(k\)-point (see Corollary~\ref{cor:intro-inline}).

The question then becomes: What geometric information is captured by the constant Brauer class \([X_{w}] - [X'_{w}]\)? 
The primary goal of this paper is to answer this question for the following important class of Fano threefolds with a degree \(4\) conic bundle structure.
Let \(\Delta\subset \PP^2\) be a smooth plane quartic and let \(\varpi\colon \Deltatilde\to \Delta\) be a smooth projective geometrically integral \'etale double cover.  Then there exists a double cover \(Y = Y_{\Deltatilde/\Delta}\to \PP^1\times \PP^2\) branched over a \((2,2)\)-divisor such that the second projection has the structure of a (geometrically standard) conic bundle with discriminant cover \(\Deltatilde\to \Delta\) (see Section~\ref{sec:obstruction-to-sections} for details).  The first projection \(\piQuadricSurface\) realizes \(Y_{\Deltatilde/\Delta}\) as a quadric surface fibration; in particular, \(Y_{\Deltatilde/\Delta}\) is rational as soon as \(\piQuadricSurface\) has a section.

\begin{thm}\label{thm:intro-mainthm-1}
    Over a field \(k\) of characteristic different from \(2\), let \(\varpi\colon \Deltatilde\to\Delta\) be a geometrically integral \'etale double cover of a smooth plane quartic and let \(X \colonequals Y_{\Deltatilde/\Delta}\) be the conic bundle described above. Assume that the Prym curve \(\Gamma_{\Deltatilde/\Delta}\) has a \(k\)-point and that the IJT obstruction for \(X \) vanishes.
    Then there is a conic bundle \(\psi \colon X' \to \PP^2\) with discriminant cover \(\Deltatilde/\Delta\), such that 
    \begin{enumerate}
        \item The threefold \(X'\) is rational, and 
        \item    \(
    [X_{\eta}] - [X'_{\eta}]\in \Br k[2]
    \) 
    is trivial if and only if \(\piQuadricSurface\colon X\to\mathbb P^1\) has a section.
    \end{enumerate}
\end{thm}
Since \( [X_{\eta}] - [X'_{\eta}]\) can be non-trivial in general, the conic bundles
\(X\to\mathbb P^2\) and \(X'\to\mathbb P^2\) in Theorem~\ref{thm:intro-mainthm-1} may be different. 
In fact, as shown in \cite{FJSVV}, the vanishing of the IJT obstruction for \(X\) does not imply rationality of \(X\).

Theorem~\ref{thm:intro-mainthm-1} raises the question of whether conic bundles \(X\to \PP^2\), \(X'\to \PP^2\) that have the same discriminant cover, but whose generic fibers are not isomorphic, can \emph{both} be rational. In other words, can \(\PP^3\) have more than one isomorphism class of conic bundle structure with the same discriminant cover? If not, then Theorem~\ref{thm:intro-mainthm-1} would imply that \(Y_{\Deltatilde/\Delta}\to \PP^1\times \PP^2\) is rational if and only if \(\piQuadricSurface\) has a section.  In general, this question seems difficult to answer, but we study it in two (non-disjoint) cases: when \(\Deltatilde\) has index \(1\) and when \(k = \RR\).
\begin{theorem}[{Proposition~\ref{prop:DeltatildeNonempty} and Theorem~\ref{thm:MultipleConicBundleStrOverR}}]\label{thm:IntroMultipleConicBundles}
    Let \(k\) be a field of characteristic different from \(2\), let \(\varpi\colon \Deltatilde\to \Deltatilde\) be a geometrically integral \'etale double cover of a smooth plane curve, and let \(\pi\colon X\to \PP^2\) and \(\pi'\colon X \to \PP^2\) be conic bundles with  discriminant cover \(\Deltatilde/\Delta\). Assume at least one of the following conditions:
    \begin{enumerate}
        \item \(\Deltatilde\) has index \(1\), i.e., there is an odd degree extension \(F/k\) with \(\Deltatilde(F) \neq\emptyset\); or \label{it:MultipleConicBundleStrOverR-part-1}
        \item \(k = \RR\), the sets \(X(\RR)\) and \(X'(\RR)\) are nonempty and connected, and \(\Delta(\RR)\subset\PP^2(\RR)\) is not a single homotopically trivial simple closed curve (i.e., not a single \defi{oval}).\label{it:MultipleConicBundleStrOverR-part-2}
    \end{enumerate}
    Then the generic fibers of \(\pi\) and \(\pi'\) are isomorphic, i.e., \([X_{\eta}] = [X'_{\eta}] \in \Br \kk(\PP^2)\).
\end{theorem}
\begin{remarks}\hfill
\begin{enumerate}
\item Condition \eqref{it:MultipleConicBundleStrOverR-part-1} holds if \(k=\RR\) and \(\deg\Delta\) is odd by Proposition~\ref{prop:ImageOfRealPoints} (this can also be deduced from the Faddeev exact sequence applied to \(X\) restricted to a line).
\item Over \(\RR\), any smooth projective rational variety must have connected and nonempty real points~\cite{DelfsKnebusch}*{Theorem~13.3}. Thus the above theorem implies that if \(\PP^3_{\RR}\) has two nonisomorphic conic bundle structures with the same discriminant cover, then the real points of the discriminant curve must be a single oval and \(\Deltatilde(\RR) = \emptyset\); this can only occur when \(\deg\Delta = 4\)~\cite{beauville-ij}*{Th\'eor\`eme 4.9}.
\item When \(\Delta\) is a smooth plane quartic,  the condition that \(\Deltatilde\) has index \(1\) is, to the best of our knowledge, the only known sufficient rationality criterion that depends \emph{only} on the discriminant cover (see \cite{FJSVV}*{Section 8.1} and Corollary~\ref{cor:deg-4-index-1-deltatilde}). 
\end{enumerate}
\end{remarks}

We prove Theorem~\ref{thm:intro-mainthm-1} by constructing \(X'\) as a smooth complete intersection of two quadrics. As a byproduct of this explicit construction, we are also able to give an explicit description of the constant Brauer class \([X_{\eta}] - [X'_{\eta}]\). This explicit constant Brauer class does not directly connect to the Brauer class that obstructs the existence of a section. To connect the two, we construct another Brauer class on one of the intermediate Jacobian torsors (denoted \(\Ptildem{1}\)) of \(X\). When this torsor has a rational point, we prove that the Brauer class on \(\Ptildem{1}\) is constant and agrees with both the Brauer class obstructing a section of \(\pi_1\) and (when \(\Gamma_{\Deltatilde/\Delta}(k)\neq\emptyset\)) the Brauer class \([X_{\eta}] - [X'_{\eta}]\). Using this Brauer class on \(\Ptildem{1}\), we obtain the following result over local fields\footnote{The archimedean local fields are \(\RR\) and \(\C\) and the nonarchimedean local fields are all finite extensions of \(\Q_p\) or \(\F_p((t))\).} when \(\Gamma_{\Deltatilde/\Delta}\) has index \(2\).

\begin{thm}[Corollary~\ref{cor:localfield_index2}]
    Let \(k\) be a local field of characteristic different from \(2\) and let \(\Deltatilde/\Delta\) be a geometrically integral \'etale double cover of a smooth plane quartic.  If \(\Gamma_{\Deltatilde/\Delta}\) has index \(2\), then \(Y_{\Deltatilde/\Delta}\) is rational if and only if the IJT obstruction vanishes.
\end{thm}

Over the real numbers, we leverage this Brauer class on \(\Ptildem{1}\) together with the topology of the real points of the conic bundle to prove the following rationality characterization when the real locus of the discriminant curve is not a single oval.
\begin{thm}\label{thm:realmain}
Let \(k = \RR\) and and let \(\Deltatilde/\Delta\) be a geometrically integral \'etale double cover of a smooth plane quartic.  Assume that \(\Delta(\RR)\) is not a single oval.  Then \(Y_{\Deltatilde/\Delta}\) is rational if and only if \(\piQuadricSurface\) has a section.
\end{thm}
For further equivalent rationality conditions over \(\RR\), see Theorem~\ref{thm:real} (which includes Theorem~\ref{thm:realmain} as a special case) and Proposition~\ref{prop:real-equiv-rationality-criteria}.

\subsection{Outline}  We begin in Section~\ref{sec:discriminant_covers} with a brief review of conic bundles, focusing on the discriminant cover \(\varpi\colon \Deltatilde\to \Delta\) of a conic bundle and the possible isomorphism classes of the generic fiber.  In this section we prove Theorem~\ref{thm:IntroMultipleConicBundles}\eqref{it:MultipleConicBundleStrOverR-part-1}.
In Section~\ref{sec:etale_double_covers} we review results from~\cite{FJSVV} that rephrase the IJT obstruction for conic bundles whose discriminant cover is a geometrically integral \'etale double cover of a smooth plane quartic in terms of torsors under the Prym variety of \(\Deltatilde\to \Delta\); in doing so, we also give equations for the Prym curve \(\Gamma_{\Deltatilde/\Delta}\).  In Section~\ref{section:reconstruction}, we show that under the assumption that the Prym curve has a rational point, there exists a conic bundle with discriminant cover \(\Deltatilde\to \Delta\) that is the blow-up along a smooth conic of a complete intersection of two quadrics. In Section~\ref{sec:alpha}, we construct a conic bundle over one of the Prym torsors from Section~\ref{sec:IJTBackground} and show that if the torsor is trivial, then this conic bundle is constant.  This construction is used in Section~\ref{sec:obstruction-to-sections} to study the obstruction to the existence of a section to \(\pi_1\colon Y_{\Deltatilde/\Delta}\to \PP^1\).  These ingredients are combined in Section~\ref{sec:proof} to prove Theorem~\ref{thm:intro-mainthm-1}.  We end in Section~\ref{sec:real} by specializing to the case \(k=\RR\) and proving Theorems~\ref{thm:IntroMultipleConicBundles}\eqref{it:MultipleConicBundleStrOverR-part-2} and~\ref{thm:realmain}.

%%%%%%%%%%%%%%%%%%%%%%%%%%%%%%%%%%%%%%%%%%%%%%%%%%%%%%%%%%%%%%%%%%%%%%%%%%%%%%%%
\section*{Acknowledgements}
%%%%%%%%%%%%%%%%%%%%%%%%%%%%%%%%%%%%%%%%%%%%%%%%%%%%%%%%%%%%%%%%%%%%%%%%%%%%%%%%
We thank David Saltman for an intriguing question and helpful follow-up discussion that motivated this project, and Asher Auel, Brendan Hassett, and Olivier Wittenberg for helpful comments. We also thank the anonymous referees for their comments and corrections which have improved the paper.

This material is based partially upon work supported by National Science Foundation grant DMS-1928930 while the last two authors were in residence at the Simons Laufer Mathematical Sciences Institute in Berkeley, California, during the Spring 2023 semester. 
S.F. was supported in part by an AMS-Simons travel grant, NSF grant DMS-2607398, and by the Hausdorff Research Institute for Mathematics, funded by the Deutsche Forschungsgemeinschaft (DFG, German Research Foundation) under Germany's Excellence Strategy – EXC-2047/1 – 390685813. L.J. was supported in part by NSF MSPRF grant DMS-2202444, NSF grant DMS-2501990, and gift SFI-MPS-TSM-00013959 from the Simons Foundation. S.S. was supported by the Dutch Research Council (NWO) grant OCENW.XL21.XL21.011. B.V. was supported in part by an AMS Birman Fellowship and NSF grant DMS-2101434. I.V. was supported in part by NSF grants DMS-2200655 and DMS-2338345. The computer algebra software \texttt{Magma}~\cite{Magma} was used to verify algebraic manipulations during proofs. For the purpose of open access, a CC BY-NC-ND public copyright license is applied to any Author Accepted Manuscript version arising from this submission.

%%%%%%%%%%%%%%%%%%%%%%%%%%%%%%%%%%%%%%%%%%%%%%%%%%%%%%%%%%%%%%%%%%%%%%%%%%%%%%%%%%%%%%%%%%%%
    \subsection*{Notation}
%%%%%%%%%%%%%%%%%%%%%%%%%%%%%%%%%%%%%%%%%%%%%%%%%%%%%%%%%%%%%%%%%%%%%%%%%%%%%%%%%%%%%%%%%%%%

    Throughout, \(k\) denotes a field of characteristic different from \(2\). Fix an algebraic closure \(\overline{k}\) of \(k\).
    A variety over \(k\) is a separated \(k\)-scheme of finite type, and a curve is a variety of pure dimension one.
    If \(Y\) is an integral \(k\)-variety, \(\kk(Y)\) denotes its function field. More generally, if \(Y\) is a
finite union of integral \(k\)-varieties \(Y_i\), then \(\kk(Y) \colonequals \prod_i \kk(Y_i)\) is the ring of global sections
of the sheaf of total quotient rings.

    For a smooth variety \(Y\) over a field \(k\), we write \(\Br Y\) for the (cohomological) Brauer group \(\HH^2_{\et}(Y, \G_m)\).  If \(Y= \Spec R\) is affine we write \(\Br R \colonequals \Br \Spec R\).  For a field \(k\), the group \(\Br k\) can also be viewed as describing the Morita equivalence classes of Severi--Brauer varieties over \(k\).  Given a smooth conic \(C\) over \(k\), we write \([C]\in \Br k[2]\) for the associated Brauer class.  For two smooth conics \(C\) and \(C'\) over \(k\), we have \(C \simeq C'\) if and only if \([C] = [C']\).
    
    For a quadratic form \(Q\in k[\mathbf{x}]\), the symmetric matrix \(M\) associated to \(Q\) is defined to be the matrix such that \(\mathbf{x}M\mathbf{x}^T = Q\); the rank of the quadratic form \(Q\) is the rank of the matrix \(M\). Given a \(\rank(Q)\)-dimensional linear space \(L\) for which \(M|_L\) has full rank, the isomorphism class of the quadratic form \(\mathbf{x}(M|_L)\mathbf{x}^T\) is independent of the choice of \(L\). The \defi{discriminant} of \(Q\) is defined by \(\disc(Q) = (-1)^{n \choose 2} \det(M|_L) \in k^{\times}/k^{\times2}\), where \(n = \dim_k L\) is the rank of \(Q\). This square class is independent of the choice of \(L\).
    In this way, a rank \(3\) quadric defines a smooth plane conic \(\mathbf{x}(M|_L)\mathbf{x}^T\), which has an associated class in \(\Br k [2]\).  Given a rank \(2\) quadric, we may associate the Brauer class given by the conic \(V(z^2 - \mathbf{x}(M|_L)\mathbf{x}^T) \subset \PP^2\). 
    
    For \(f, g \in \kk(\PP^2)\), we write \((f, g)\) for the corresponding class in \(\Br \kk(\PP^2)[2]\). If instead, \(F, G\) are homogeneous ternary forms of even degree, then any non-trivial linear form \(\ell\) defines a class \((F/\ell^{\deg{F}}, G/\ell^{\deg{G}}) \in \Br \kk(\PP^2)[2]\). Since this class is independent of the choice of \(\ell\), by slight abuse of notation, we denote it by \((F,G)\).
    
    The Brauer group \(\HH^2_{\et}(-, \G_m)\) is a contravariant functor, and given a morphism of smooth varieties \(f\colon Y \to Z\) we write \(\alpha_Y \colonequals f^*\alpha\) for \(\alpha\in \Br Z\).  Given any point \(z\in Z\) we write \(Y_z\) for the fiber of \(f\) over \(z\).  

    For a smooth connected curve \(C\),
    we write \(\bPic_C\) for the \defi{Picard scheme} of \(C\) and write \(\bPic_C^i\) for the connected component that parametrizes degree \(i\) line bundles.

    Throughout the paper, we will write \(\eta\) for the generic point of \(\PP^2\).

%%%%%%%%%%%%%%%%%%%%%%%%%%%%%%%%%%%%%%%%%%%%%%%%%%%%%%%%%%%%%%%%%%%%%%%%%%%%%%%%%%%%%
\section{Conic bundles and discriminant covers}\label{sec:discriminant_covers}
%%%%%%%%%%%%%%%%%%%%%%%%%%%%%%%%%%%%%%%%%%%%%%%%%%%%%%%%%%%%%%%%%%%%%%%%%%%%%%%%%%%%%

    The central objects of interest in this paper are \defi{conic bundles} over \(\PP^2\), i.e., the class of smooth projective varieties \(X\) with a flat morphism \(\pi\colon X \to \PP^2\) of relative dimension \(1\) such that \(\omega_X\) is relatively anti-ample. This implies that locally over \({\PP^2}\), the threefold \(X\) is defined by a quadratic form and that the fiber of \(\pi\) over any \(w\in \PP^2\) is a conic in \(\PP^2\) (possibly of rank less than \(3\) if \(w\) is not the generic point). 
    These and other background results on conic bundles can be found in \cite{ProkhorovSurvey}*{Section 3}  and \cite{Sarkisov}*{Section 1}.
    
    There is a curve \(\Delta \subset \PP^2\) with the property that the fiber over \(w\in \PP^2\) is a smooth conic if and only if \(w\not\in\Delta\). 
    Furthermore, there is a curve \(\Deltatilde\) that is a double cover of \(\Delta\) and that parametrizes the irreducible components of the fibers of \(\pi \colon X_{\Delta}\to \Delta\). The curve \(\Deltatilde\) can be obtained by considering the Stein factorization of the relative variety \(\mathcal F_1(X_{\Delta}/\Delta)\to \Delta\subset \PP^2\) of lines in the fibers of \(\pi\). We call \(\Delta\) the \defi{discriminant locus} of \(\pi\) and \(\varpi\colon \Deltatilde\to \Delta\) the \defi{discriminant cover} of \(X\to \PP^2\). If \(\Delta\) is smooth, then every fiber of \(X \to \PP^2\) has rank at least \(2\), and so \(\varpi\) is \'etale. 
    
    A conic bundle over \(k\) is \defi{standard} if \(\Pic X = \pi^*\Pic \PP^2 \oplus \Z\); this is equivalent to \(\Deltatilde\) and \(\Delta\) having the same number of irreducible components. Throughout the paper, we reserve \(\pi\colon X \to \PP^2\) for a geometrically standard conic bundle, i.e., the base change \(X_\kbar \to \PP^2_\kbar\) is standard, which implies that \(X\to \PP^2\) is standard as well.

%%%%%%%%%%%%%%%%%%%%%%%%%%%%%%%%%%%%%%%%%%%%%%%%%%%%%%%%%%%%%%%%%%%%%%%%%%%%%%%%%%%%%%%%%%%%
    \subsection{Discriminant covers and isomorphism classes of conic bundles}
%%%%%%%%%%%%%%%%%%%%%%%%%%%%%%%%%%%%%%%%%%%%%%%%%%%%%%%%%%%%%%%%%%%%%%%%%%%%%%%%%%%%%%%%%%%%

    Discriminant covers are crucial to the study of conic bundles and control much of their geometry.  Indeed, over fields \(k\) with \(\Br k[2] = 0\), the discriminant cover determines the isomorphism class of \(X_{\eta}\).  Over arbitrary fields, we study the isomorphism class of \(X_{\eta}\) via its class \([X_\eta]\) in \(\Br \kk(\PP^2)[2]\) by considering the purity exact sequence~\cite{CTS-Brauer-book}*{Theorems 3.7.3 and 6.1.3}
    \[
    0 \longrightarrow \Br k[2] \longrightarrow \Br \kk(\PP^2)[2] \xrightarrow{\;(\partial_z)_z\;} \bigoplus_{z \in (\PP^2)^{(1)}} H^1(\kk(z), \Z/2\Z),
    \]
    where \((\PP^2)^{(1)}\) denotes the set of irreducible divisors in \(\PP^2\).
    By definition of the residue maps \(\partial_z\)~\cite{CTS-Brauer-book}*{Definition 1.4.11}, the residue \(\partial_z([X_{\eta}])\) is trivial for every \(z\) outside of the support of \(\Delta\). Furthermore, if \(\kk(\Deltatilde)_{z}\) denotes the function field of the component of \(\Deltatilde\) lying over \(z\), 
    then \(\partial_z([X_{\eta}]) = \kk(\Deltatilde)_{z}\) for every \(z\in (\PP^2)^{(1)}\) in the support of \(\Delta\).  Thus, if \(\pi\colon X \to \PP^2\) and \(\pi'\colon X'\to \PP^2\) are two standard conic bundles with the same discriminant cover, then
    \begin{equation}\label{eq:constant}
        [X_{\eta}] - [X'_{\eta}] \in \im (\Br k[2] \hookrightarrow \Br \kk(\PP^2)[2]).
    \end{equation}
    Further, this difference agrees with the image of the difference \([X_w] - [X'_w]\) for any \(w\in (\PP^2\smallsetminus\Delta)(k)\).
    However, as the following proposition shows, for a fixed discriminant cover,
     not every element of \(\Br k[2]\)
     can be realized as a difference of two standard conic bundles with
     this discriminant cover.
     
\begin{prop}\label{prop:DeltatildeNonempty}
    Let \(\varpi\colon \Deltatilde\to \Delta\) be an \'etale double cover of smooth, projective, geometrically integral curves with \(\Delta\) a plane curve over a field \(k\) of characteristic different from \(2\).  If \(\Deltatilde\) has index \(1\), i.e., if \(\Deltatilde(F)\neq\emptyset\) for some odd degree extension \(F/k\),
    then any conic bundles \(\pi\colon X \to \PP^2\), \(\tilde\pi\colon \tilde{X} \to \PP^2\) with discriminant cover \(\Deltatilde/\Delta\) must have isomorphic generic fibers. That is, \([X_\eta]-[\tilde{X}_\eta] =0 \in\Br\kk(\PP^2)\).
\end{prop}
\begin{proof}
    Let \(\pi\colon X \to \PP^2\), \(\tilde\pi\colon \tilde{X} \to \PP^2\) be  conic bundles with discriminant cover \(\Deltatilde/\Delta\).  Then \([X_{\eta}] - [\tilde{X}_{\eta}] \in \im (\Br k[2] \to \Br \kk(\PP^2)[2])\) by~\eqref{eq:constant}.  We will show that for any nontrivial \(\alpha\in \Br k[2]\), the class \([X_{\eta}] + \alpha\) has no quadratic splitting field, which implies that \([X_{\eta}] + \alpha\) is \emph{not} represented by a conic.  Thus, \([X_{\eta}] - [\tilde{X}_{\eta}]\) must equal \(0\in \Br\kk(\PP^2)\), as desired.

    The statement that \([X_{\eta}] + \alpha\) has no quadratic splitting field is equivalent to the statement that the index of \([X_{\eta}] + \alpha\) is \(2^n\) for some \(n>1\). Since the index of a \(2\)-primary Brauer class is unchanged under odd degree extensions~\cite{GS}*{Corollary 4.5.11}, we may reduce to the case that \(\Deltatilde(k)\neq\emptyset\).

    We now return to proving that \([X_{\eta}] + \alpha\) has no quadratic splitting field, assuming that \(\Deltatilde(k)\neq\emptyset\).  By a result of Albert~\cite{Lam-quadratic-forms}*{Chapter III, Theorem 4.8}, \([X_{\eta}] + \alpha\) has a quadratic splitting field if and only if \(\alpha\) and \([X_{\eta}]\) have a common quadratic splitting field \(L/\kk(\PP^2)\).  Assume that there exists such an \(L\).  Let \(R\) denote the local ring \(\OO_{\PP^2, \Delta}\) and identify \(\kk(\PP^2) = \Frac(R)\); let \(S\) denote the integral closure of \(R\) in \(L\).  Since \(L\) splits \(X_{\eta}\) and \(\pi\colon X \to \PP^2\) is proper, the valuative criterion for properness gives a morphism \(\Spec S \to X_R\).  If \(R\hookrightarrow S\) is unramified, then the image of \(\Spec S \to X_R\) cannot meet the singular locus of the special fiber of \(X_R\to \Spec R\).  In particular, since \(X_{\kk(\Delta)}\) is \(\kk(\Delta)\)-birational to \(\PP^1_{\kk(\Deltatilde)}\) (by definition of the discriminant cover), 
    if \(R\hookrightarrow S\) is an unramified extension, then the extension of residue fields agrees with \(\kk(\Deltatilde)/\kk(\Delta)\). Thus, we have the following commutative diagrams:
\begin{center}
    \begin{tikzcd}
    L \arrow[r, hookleftarrow] &S \arrow[r]  & \kk(\Deltatilde) \\
    &R  \arrow[u, hook] \arrow[r] & \kk(\Delta)\arrow[u, hook]\\
    &k \arrow[u, hook] \arrow[ur, hook]
    \end{tikzcd}
    \(\quad\)and\(\quad\)
    \begin{tikzcd}
    \Br L \arrow[hookleftarrow, r]& \Br S \arrow[r] & \Br \kk(\Deltatilde) \\
    & \Br R  \arrow[u] \arrow[r] & \Br \kk(\Delta)\arrow[u]\\
     &  \Br  k \arrow[u, hook] \arrow[ur, hook]
    \end{tikzcd}.
\end{center}
    By assumption \(L\) splits \(\alpha\), so the diagram on the right shows that \(\alpha\) becomes trivial when pulled back to \(\kk(\Deltatilde)\).  The pullback map factors as \(\Br k \to \Br \Deltatilde \hookrightarrow \Br \kk(\Deltatilde)\).  Since \(\Deltatilde(k)\neq\emptyset\), \(\Br k \to \Br \Deltatilde\) is injective, 
    so \(\alpha=0\in\Br k\).

    In the case that \(R\hookrightarrow S\) is ramified, then the residue field of \(S\) is \(\kk(\Delta)\).  Since \(\Deltatilde\) is a cover of \(\Delta\), we have \(\Delta(k)\neq\emptyset\) and so we can repeat the above argument with \(\Delta\) replacing \(\Deltatilde\).
    \end{proof}

In the case that \(\Delta\) has degree \(4\), we prove the following strengthening of \cite{FJSVV}*{Proposition 8.1(ii)}.

\begin{cor}\label{cor:deg-4-index-1-deltatilde}
    In the setting of Proposition~\ref{prop:DeltatildeNonempty}, assume \(\Delta\) is a smooth plane quartic and \(\Deltatilde\) has index \(1\). Then any conic bundle \(X\to\PP^2\) with discriminant cover \(\Deltatilde/\Delta\) is rational.
\end{cor}

\begin{proof}
    By \cite{FJSVV}*{Section 6} there exists a threefold \(Y\) with the structure of a conic bundle with discriminant cover \(\Deltatilde/\Delta\) and the structure of a quadric fibration \(\piQuadricSurface\colon Y \to \PP^1\).
    \cite{FJSVV}*{Proposition 6.1(v)} and the assumption on \(\Deltatilde\) imply that there is an odd degree extension \(L/k\) over which \(({\piQuadricSurface})_L\colon Y_L \to \PP^1_L\) has a section. Then \(\piQuadricSurface\) has a section over \(k\) by Springer's theorem \cite{Springer52}, so \(Y\) is rational. Next, let \(X\to\PP^2\) be any conic bundle with discriminant cover \(\Deltatilde/\Delta\). Then Proposition~\ref{prop:DeltatildeNonempty} implies that \(X\) and \(Y\) are birationally equivalent, and therefore \(X\) is rational.
\end{proof}

%%%%%%%%%%%%%%%%%%%%%%%%%%%%%%%%%%%%%%%%%%%%%%%%%%%%%%%%%%%%%%%%%%%%%%%%%%%%%%%%%%%%%%%%%%%%
    \subsection{Prym varieties and the IJT obstruction to rationality}\label{sec:etale_double_covers}
%%%%%%%%%%%%%%%%%%%%%%%%%%%%%%%%%%%%%%%%%%%%%%%%%%%%%%%%%%%%%%%%%%%%%%%%%%%%%%%%%%%%%%%%%%%%

    The primary conic bundles \(\pi\colon X \to \PP^2\) of interest in this paper are those whose discriminant curve \(\Delta\subset \PP^2\) is a smooth plane curve (often of degree \(4\)) and whose discriminant cover \(\varpi\colon \Deltatilde \to \Delta \) is geometrically integral.  The smoothness of \(X\) and \(\Delta\) imply that \(\varpi\) is an \'etale double cover.

    The \defi{Prym variety} of \(\Deltatilde/\Delta\), which we denote by \(P\), is the connected component of the identity of \(\ker(\piDelta_* \colon \bPic_{\Deltatilde/k} \to \bPic_{\Delta/k})\); it is a principally polarized abelian variety (see \cite{mumford-prym}). The Prym variety is the intermediate Jacobian of \(X\), and can be thought of as a parameter space for algebraically trivial curve classes on \(X\)~\citelist{\cite{beauville-ij}*{Proposition 3.3}, \cite{FJSVV}*{Theorem 5.8}}. The kernel of \(\piDelta_* \colon \bPic_{\Deltatilde/k} \to \bPic_{\Delta/k}\) has a unique additional connected component, which we denote \(\Ptilde\), and it is a torsor under \(P\) (see \cite{mumford-prym}*{Sections 2 and 6}).

    \subsubsection{\'Etale double covers of smooth plane quartics}\label{sec:IJTBackground}
    If \(\Delta\subset \PP^2\) is a smooth plane quartic, then any geometrically integral \'etale double cover \(\Deltatilde\) is a smooth genus \(5\) curve.  Bruin \cite{bruin} proved that the Prym variety of such a double cover is always isomorphic to the Jacobian of a smooth genus \(2\) curve \(\Gamma = \Gamma_{\Deltatilde/\Delta}\).  This implies that on the threefold \(X\), the Clemens--Griffiths intermediate Jacobian obstruction to rationality~\cites{clemensgriffiths,bw-cg} vanishes. 
    Defining equations of \(\Delta, \Deltatilde, \Gamma\) can be given in terms of 3 quadrics \(Q_1,Q_2,Q_3\in k[u,v,w]\) as follows:
    \begin{align}
        \Delta  &=V(Q_2^2 - Q_1Q_3), \label{eq:Delta}\\
        \Deltatilde& = V(Q_1 - r^2, Q_2 - rs, Q_3 - s^2)\subset \PP^4,\textup{and} \label{eq:Deltatilde} \\
        \Gamma &= V(y^2 + \det(t_0^2M_1 + 2t_0t_1M_2 + t_1^2M_3))\subset \PP(1,1,3),\label{eq:Gamma}
    \end{align}
    (here \(M_i\) is the symmetric matrix associated to \(Q_i\)); see, for example,~\cite{bruin} and \cite{FJSVV}*{Remark 4.7}.  In addition, \(Q_1,Q_2,Q_3\) are uniquely determined up to a \(\PGL_2\)-action~\cite{FJSVV}*{Theorem 4.5(i)} which agrees with changing coordinates of \(t_0,t_1\) on \(\Gamma\). Precisely:
    \[
        \begin{pmatrix}a & b\\c & d\end{pmatrix}\cdot (Q_1,Q_2,Q_3)\mapsto
        (a^2Q_1 + 2acQ_2 + c^2Q_3, abQ_1 + (ad + bc)Q_2 + cdQ_3, b^2Q_1 + 2bdQ_2 + d^2Q_3) .
    \]
    \begin{remark}
        It is important to note that the cover \(\Deltatilde \to \Delta\) depends on the \emph{forms} \(Q_1,Q_2,Q_3\), not only the conics they define. Indeed if \(Q_1,Q_2,Q_3\) are scaled by a constant, then \(\Deltatilde\) changes by a quadratic twist. 
    \end{remark}

    The \defi{intermediate Jacobian torsor (IJT) obstruction} for \(X\), introduced by \cite{HT-intersection-quadrics} over \(\RR\) and \cite{bw-ij} over arbitrary \(k\) (see also \cite{HT-cycle} for \(k\subset\mathbb C\)), is an obstruction to rationality of \(X\) that is encoded by certain torsors under \(P\). In the case where \(\Delta\subset \PP^2\) is degree \(4\), the obstruction is encoded by \(\Ptilde\) and the following two \(P\)-torsors:
    \begin{align}
    \Pm{1} &\colonequals \{ L \in \bPic^4_{\Deltatilde/k} \colon \piDelta_*L \simeq \OO_{\Delta}(1), \hh^0(L) \equiv 0\bmod 2 \}, \;\textup{and}\label{def:P1}\\
    \Ptildem{1}  &\colonequals \{ L \in \bPic^4_{\Deltatilde/k} \colon \piDelta_*L \simeq \OO_{\Delta}(1), \hh^0(L) = 1 \}.\label{def:Ptilde1}
\end{align}
\begin{prop}\label{prop:IJTforConicBundles}
    Let \(k\) be a field of characteristic different from \(2\), and let \(\pi\colon X\to \PP^2\) be a conic bundle over \(k\) with discriminant cover \(\Deltatilde \to \Delta\) such that \(\Delta\) is a smooth quartic and \(\Deltatilde\) is geometrically integral. 
    \begin{enumerate}
        \item{\cite{bruin}*{Section 5, Case 4}} As \(P\)-torsors, \(\Pm{1}\simeq \bPic^1_{\Gamma}\).
        \item{\cite{FJSVV}*{Corollary of Theorem 6.4, Section 6.4}} The IJT obstruction vanishes for \(X\) if and only if either \(\Ptilde(k) \neq \emptyset\) or \(\Ptildem{1}(k) \neq \emptyset\). Moreover, if \(\bPic_{\Gamma}^1(k) \neq \emptyset\), then 
    \(
    \Ptildem{1}(k) \neq \emptyset \) if and only if \(\Ptilde(k) \neq \emptyset
    \),
    so in this case the vanishing of the IJT obstruction for \(X\) is equivalent to \(\Ptildem{1}(k) \neq \emptyset\).
    \end{enumerate}
\end{prop}
\begin{remark}
The torsors \(\Ptilde, \Pm{1}, \Ptildem{1}\) are part of a larger family of torsors, given by the  polarized Prym scheme, which can be defined for arbitrary conic bundles; see \cite{FJSVV}*{Section 4} for more details.
\end{remark}

%%%%%%%%%%%%%%%%%%%%%%%%%%%%%%%%%%%%%%%%%%%%%%%%%%%%%%%%%%%%%%%%%%%%%%%%%%%%%%%%
\section{Constructing an intersection of quadrics with a  conic bundle structure}
\label{section:reconstruction}
%%%%%%%%%%%%%%%%%%%%%%%%%%%%%%%%%%%%%%%%%%%%%%%%%%%%%%%%%%%%%%%%%%%%%%%%%%%%%%%%

The goal of this section is to prove the following proposition.
\begin{prop}
\label{prop:construction-intersection-of-quadrics}
    Let \(k\) be a field of characteristic different from \(2\).
    
    Given quadrics \(Q_1, Q_2,Q_3\in k[u,v,w]\) such that \(V(Q_2^2-Q_1 Q_3)\) is a smooth quartic and such that either \(Q_1\) is rank \(2\) or \(\disc(Q_1)\) is a square, there exist two quadrics \(q_0,q_\infty\in k[x_0,x_1, \dots, x_5]\) such that
    \begin{enumerate}
        \item The intersection \(Z\colonequals V(q_0, q_\infty)\subset \PP^5\) is a smooth projective threefold;
        \item The linear section \(Z\cap V(x_3,x_4,x_5)\) is a smooth conic \(C\); \label{part:smoothconic}
        \item Projection away from the plane \(V(x_3,x_4,x_5)\) induces a morphism \(\Bl_C Z \to \PP^2\), endowing \(\Bl_C Z\) with a geometrically standard conic bundle structure whose generic fiber equals
        \[
            [(\Bl_C Z)_{\eta}] = (Q_1, Q_2^2 - Q_1Q_3) + [Q_1] \in \Br \kk(\PP^2),
        \]
        where \([Q_1]\) denotes the pullback of the class in \(\Br k[2]\) corresponding to the rank \(2\) or \(3\) quadric \(Q_1\).
        In particular, the discriminant cover of the conic bundle \(\Bl_CZ \to \PP^2\) is \(\Deltatilde/\Delta\), where \(\Delta\) and \(\Deltatilde\) are as defined in Section~\ref{sec:IJTBackground}; and\label{part:genericfiber}
        \item The hyperelliptic curve \(\Gamma_Z\) arising in the Stein factorization \(\calF_2(\Bl_Z \PP^5/\PP^1) \to \Gamma_Z \to \PP^1\) of the Fano scheme of \(2\)-planes in the pencil of quadrics \(V(t_0q_\infty - t_1q_0) \subset \PP^5 \times \PP^1\) is isomorphic to \(\Gamma_{\Deltatilde/\Delta}\). \label{part:hyperelliptic}
    \end{enumerate}
    \end{prop}

\begin{remarks}\hfill
\begin{enumerate}
    \item The condition that either \(Q_1\) is rank \(2\) or \(\disc(Q_1)\) is a square guarantees that there is a \(k\)-point of \(\Gamma_{\Deltatilde/\Delta}\) lying over \([1:0]\in \PP^1\) by~\eqref{eq:Gamma}. 
    Conversely, if \(\Gamma_{\Deltatilde/\Delta}(k)\neq\emptyset\), then by an automorphism of \(\mathbb{P}^1\), we may assume that one of the rational points lies above \([1:0]\) and then (after this change of coordinates)  either \(Q_1\) will be rank \(2\) or \(\disc(Q_1)\) will be a square.
   
    \item By~\eqref{part:hyperelliptic}, the condition
    \(\Gamma_{\Deltatilde/\Delta}(k) = \Gamma_Z(k) \neq \emptyset\) is clearly necessary, since the \(2\)-plane spanned by a conic contained in \(Z\) is necessarily contained in some quadric in the pencil.  In general, a hyperelliptic curve without a point may arise from a pencil of quadrics; see \cite{BhargavaGrossWang}*{Theorem 24}. However, such pencils of quadrics do not have a conic bundle structure of the form described in the proposition.
    
    \item
      Our proof is constructive. For~\eqref{part:hyperelliptic} to hold, the rank of $q_{\infty}$ must be equal to \(3 + \rk (Q_1)\). Additionally, for~\eqref{part:smoothconic} to hold together with~\eqref{part:hyperelliptic}, the quadric \(V(q_{\infty})\) must contain \(V(x_3,x_4,x_5)\). Thus, if \(q_{\infty}\) is rank \(6\), then, after a possible change of coordinates we may assume \(q_\infty = x_0x_3 + x_1x_4 + x_2x_5\). Similarly, if \(q_\infty\) has rank \(5\), then we may assume it is equal to \(x_3^2 + x_1x_4 + x_2x_5\). The conditions on \(q_0\) are more delicate. Guided by computations of the discriminant double cover of the conic bundle associated to an intersection of quadrics, we construct a viable \(q_0\) and then prove that this construction works.
\end{enumerate}
\end{remarks}

Using Proposition~\ref{prop:construction-intersection-of-quadrics} together with work of Hassett--Tschinkel and Benoist--Wittenberg, we can deduce the assertion from the introduction.
\begin{cor}[Corollary of Proposition~\ref{prop:construction-intersection-of-quadrics} \&~\citelist{\cite{HT-intersection-quadrics}*{Theorem 36}, \cite{bw-ij}*{Theorem A}}]\label{cor:intro-inline}
    Let \(k\) be a field of characteristic different from \(2\), and let \(\pi\colon X \to \PP^2\) be a conic bundle over \(k\) with discriminant cover \(\varpi\colon\Deltatilde\to \Delta\) where \(\Delta\subset \PP^2\) is a smooth plane quartic and \(\Deltatilde\) is geometrically integral.  Assume that \(\Gamma_{\Deltatilde/\Delta}(k) \neq \emptyset\).  
    
    There exists a conic bundle \(\pi'\colon X'\to \PP^2\) with discriminant cover \(\Deltatilde/\Delta\) (namely \(X' = \Bl_C Z\) with \(Z,C\) as in Proposition~\ref{prop:construction-intersection-of-quadrics}) such that \(X'\) is rational if and only if the IJT obstruction vanishes for \(X\) (equivalently for \(X'\)).
\end{cor}
\begin{proof}
    Since \(\Gamma_{\Deltatilde/\Delta}(k)\neq\emptyset\), by a change of coordinates on \(\PP^1\), we may assume that either \(Q_1\) is rank \(2\) or \(\disc(Q_1)\) is a square by~\eqref{eq:Gamma}.
    Then we may apply Proposition~\ref{prop:construction-intersection-of-quadrics} to obtain an intersection of quadrics \(Z\) containing a conic \(C\) such that the conic bundle morphism \(\pi'\colon \Bl_C Z \to \PP^2\) has discriminant cover \(\varpi\colon \Deltatilde\to \Delta\).  Then~\citelist{\cite{HT-intersection-quadrics}*{Theorem 36}, \cite{bw-ij}*{Theorem A}} imply that the IJT obstruction characterizes rationality for \(\Bl_C Z\) (see~\cite{FJSVV}*{Corollary 8.3} for details). 
\end{proof}

\begin{proof}[Proof of Proposition~\ref{prop:construction-intersection-of-quadrics}]
    The assumption that \(V(Q_2^2-Q_1Q_3)\) is a smooth quartic implies that the quadrics \(Q_1,Q_3\) both have rank at least \(2\). This furthermore implies that the curve \(\Deltatilde\) defined in~\eqref{eq:Deltatilde} is geometrically integral. We consider two cases for the rank of \(Q_1\). \texttt{Magma} code verifying all computational claims in the proof can be found in the ~\href{https://github.com/ivogt161/Constructing-an-intersection-of-quadrics-with-a-conic-bundle-structure}{Github repository} associated to this paper.

    \textbf{Case 1}: \(Q_1\) has rank \(3\), in which case \(\disc(Q_1) = -\det(M_1)\) is a square.  We change coordinates on \(\PP^2\) to assume that \(Q_1 = au^2 + bv^2 - abw^2\) for some \(a,b\in k^{\times}.\) Note that then \([Q_1] = (a,b)\in \Br k[2]\).
    Consider the following \(6\times6\) symmetric matrices defined in terms of the symmetric matrices \(M_i\) associated to \(Q_i\):
    \begin{equation}\label{eq:Zquadrics}
    A_0 \colonequals \begin{pmatrix} abM_1^{-1} & -M_1^{-1} M_2 \\ -M_2 M_1^{-1} & \frac{-1}{ab}(M_3 - M_2M_1^{-1}M_2)  \end{pmatrix}, \qquad  A_\infty \colonequals \begin{pmatrix} 0 & I \\ I & 0  \end{pmatrix}. 
    \end{equation}
    Let \(q_0, q_{\infty}\) be the two quadratic forms defined by these matrices, i.e., \(q_i = \textbf{x}A_i\textbf{x}^T\).  We claim that \(Z \colonequals V(q_0,q_{\infty})\subset \PP^5\) has the desired properties.

    An intersection of two quadrics in six variables is smooth and three-dimensional if and only if the discriminant polynomial \(-\det(TA_\infty-A_0)\) is separable~\cite{Reid-thesis}*{Proposition 2.1}. Using the Schur complement, we compute
    \begin{align*}    
    -\det( TA_{\infty} - A_0) & = \det(M_1^{-1}) \det(M_3 - M_2M_1^{-1}M_2 + (M_2 M_1^{-1} + TI)M_1(M_1^{-1} M_2 + TI))\\ 
    & = \det(M_1^{-1})\det(M_3 + 2TM_2 + T^2M_1).
    \end{align*}
Since \(\disc(Q_1) = -\det(M_1)\) is a square, by comparing with~\eqref{eq:Gamma}, we conclude that the hyperelliptic curve \(y^2 = -\det(TA_\infty-A_0)\) is isomorphic to \(\Gamma_{\Deltatilde/\Delta}\). In particular, the polynomial \(-\det(TA_{\infty}-A_0)\) is separable.  Hence, \(Z\) is smooth and three-dimensional, as desired. This also proves \eqref{part:hyperelliptic} in this case, since \(y^2 = \disc(Tq_\infty - q_0) =  -\det(TA_\infty-A_0)\) is an (affine) equation of the hyperelliptic curve \(\Gamma_Z\).

Since \(q_{\infty}\) is identically zero when restricted to \(V(x_3,x_4,x_5)\), the linear section of \(Z\) given by \( V(x_3,x_4,x_5)\) is equal to \(V(q_0,x_3,x_4,x_5)\) so by~\eqref{eq:Zquadrics} is given by the vanishing of the rank \(3\) form corresponding to \(M_1^{-1}\). Thus we have proved~\eqref{part:smoothconic}.  

To prove~\eqref{part:genericfiber}, we must further describe the conic bundle structure on \(\Bl_C Z\). Let \(\Lambda \colonequals V(x_3, x_4, x_5)\), and let \(\phi\colon Z \dasharrow \PP^2\) be the projection away from \(C\colonequals V(q_0) \cap \Lambda\).
Fix a point  \(p= [u: v: w] \in \PP^2\).  We obtain the fiber of \(\phi\) above \(p\) by taking the intersection \(\Span(\Lambda, [0:0:0:u:v:w])\cap Z\subset \PP^5\), removing \(\Lambda\), and then taking the Zariski closure.

Note that the restriction of \(q_{\infty}\) to \(\Span(\Lambda, [0:0:0:u:v:w])\) is a rank \(2\) quadratic form, and one can compute that the vanishing of this quadratic form is the union of \(\Lambda\) and the \(3\)-dimensional column span of the following \(6\times 4\) matrix
     \[
    B_{p} = \begin{pmatrix}
        -v & -w & 0 & 0\\
        u & 0 & -w & 0\\
        0 & u & v & 0\\
        & \textbf{0}_{3,3} & & [u,v,w]^T
    \end{pmatrix};
    \]  
    we denote the column span of this matrix \(\Lambda_p\).
    Thus the fiber of \(\phi\) above \(p\) is \(\Lambda_p\cap V(q_0)\), which is given by the rank \(3\) quadratic form associated to the \(4\times4\) symmetric matrix \(B_{p}^T A_{0} B_{p}\). Let  \(\mathbf{M}\) denote the bottom right \(3\times3\) submatrix of \(B_{p}^T A_{0} B_{p}\)
    (in the \texttt{Magma} code, \(\mathbf{M}\) is denoted by  `\textsf{bM}'). The determinants of the top right \(i\times i\) submatrices of \(\mathbf{M}\) are:
    \[
        -(u^2 - bw^2), \quad -w^2Q_1, \quad \frac{-w^2}{ab}(Q_2^2 - Q_1Q_3), \quad \textup{for }i=1,2, 3, \textup{ respectively}.
    \]
    In particular, \(\mathbf{M}\) is generically invertible, so it represents the generic fiber of the conic bundle. Note that a diagonalization of \(\mathbf{M}\) has entries that are successive ratios of these minors (where the \(0^{th}\) minor is said to be \(1\)), in other words
    \begin{align*}
        \mathbf{M} \sim_{\textup{conj}} & \Diag\left(-(u^2 - bw^2), \frac{-w^2Q_1}{-(u^2 - bw^2)}, \frac{-w^2}{-abw^2Q_1}(Q_2^2 - Q_1Q_3) \right).
    \end{align*}
    Since a diagonal matrix \(\Diag(d_1,d_2,d_3)\) corresponds to the Brauer class \((-d_1d_3,-d_1d_2)\), the Brauer class of the generic fiber is
    \[
        \left(ab(u^2 - bw^2)Q_1(Q_2^2 - Q_1Q_3), Q_1\right) = 
        \left(Q_2^2 - Q_1Q_3, Q_1\right) + \left(-Q_1, Q_1\right) + \left(-ab(u^2 - bw^2), Q_1\right).
    \]
    Furthermore, using the Brauer relations \((f,g^2) = 1\) and  \((f,g) =  (f, g(h^2 - f))\) for all \(f,g\in\kk(\PP^2)^{\times}\) and \(h\in \kk(\PP^2)\) \cite{CTS-Brauer-book}*{Lemma 1.1.6 and Proposition 1.1.8} and the identity \(bQ_1 = (bv)^2 - (-ab(u^2 - bw^2))\),
    we see that \(\left(-Q_1, Q_1\right)\) is trivial and that
    \begin{align*}
        (-ab(u^2 - bw^2), Q_1) & = (-ab(u^2 - bw^2), b) = (-ab, b) = (a, b),
    \end{align*}
    which completes the proof. 

    \textbf{Case 2}: \(Q_1\) has rank \(2\).  The proof will proceed similarly as in the rank \(3\) case with the specific equations modified slightly, so we will omit more of the algebraic details.
    
    We change coordinates on \(\PP^2\) to assume that \(Q_1 = av^2 + bw^2\) for some \(a,b\in k^{\times}.\) Note that then \([Q_1] = (a,b)\in \Br k[2]\). Since \(\Delta\) is smooth, \(\lambda \colonequals Q_2(1,0,0)\neq0\) and so by replacing \(Q_2\) and \(Q_3\) with \(-ab\lambda^{-1} Q_2\) and \((-ab\lambda^{-1})^2Q_3\), we may assume that \(Q_2(1,0,0) = -ab\).
    We will again consider \(6\times6\) symmetric matrices defined in terms of the symmetric matrices \(M_i\) associated to \(Q_i\), but we need some additional notation.  Let \(T_2\) denote the unique upper triangular matrix such that \(T_2 + T_2^{T} = 2M_2\).  For any \(2\times2\) matrix \(N\) and any \(c\in k\), we write \(c\oplus N\) for the \(3\times 3\) block diagonal matrix with \(c\) in the top left entry and \(N\) in the bottom right entry.  Let \(N_1\) denote the \(2\times 2\) diagonal matrix with entries \(a,b\) so that \(M_1 = 0\oplus N_1\).  Now define the \(6\times 6\) symmetric matrices
    \begin{equation}
    A_0  \colonequals
    \begin{pmatrix} 
    1\oplus -abN_1^{-1} & -(0\oplus N_1^{-1})T_2^{T} \\ 
    -T_2(0\oplus N_1^{-1}) & \frac{1}{ab}(M_3 + T_2(0\oplus -N_1^{-1})T_2^{T})  
    \end{pmatrix},     \quad
      A_\infty  \colonequals 
    \begin{pmatrix} 
        \mathbf{0}_{3,3} & 0\oplus I_2\\
        0\oplus I_2 & 2 \oplus \mathbf{0}_{2,2}
    \end{pmatrix}. \label{eq:ZquadricsRank5}
    \end{equation}
    Let \(q_0, q_{\infty}\) be the two quadratic forms defined by these matrices, i.e., \(q_i = \textbf{x}A_i\textbf{x}^T\); we will show that \(Z \colonequals V(q_0,q_{\infty})\subset \PP^5\) has the desired properties.

    Using the Schur complement, noting that \((-2ab)\oplus \mathbf{0}_{2,2} + T_2(0 \oplus I_{2}) + (0 \oplus I_{2})T_2^T = 2M_2\) by our normalization \(Q_2(1,0,0) = -ab\), we have 
\(\det(TA_{\infty} - A_0) = (ab)^{-2}\det(M_3 + 2TM_2 + T^2M_1)\). Hence \(Z\) is smooth and three-dimensional, as desired. This proves~\eqref{part:hyperelliptic}.  Since \(q_{\infty}\) is identically zero when restricted to \(V(x_3,x_4,x_5)\), the linear section of \(Z\) given by \( V(x_3,x_4,x_5)\) is equal to \(V(q_0,x_3,x_4,x_5)\) and so is given by the rank \(3\) quadratic form associated to the top left \(3\times3\) submatrix of \(A_0\). This proves~\eqref{part:smoothconic}.

To prove~\eqref{part:genericfiber}, we must further describe the conic bundle structure on \(\Bl_C Z\). We again let \(\Lambda \colonequals V(x_3, x_4, x_5)\), and let \(\phi\colon Z \dasharrow \PP^2\) be the projection away from \(C\colonequals V(q_0) \cap \Lambda\).
Fix a point  \(p= [u: v: w] \in \PP^2\).  As above, the fiber of \(\phi\) above \(p\) is obtained by taking the Zariski closure of \(\left(\Span(\Lambda, [0:0:0:u:v:w])\cap Z\right)\smallsetminus\Lambda\subset \PP^5\). Note that the restriction of \(q_{\infty}\) to \(\Span(\Lambda, [0:0:0:u:v:w])\) is a rank \(2\) quadratic form, and one can compute that vanishing of this quadratic form is the union of \(\Lambda\) and the \(3\)-dimensional column span of the following matrix
     \[
    B_{p} = \left(\begin{array}{cc|cc}
        1 & 0 & 0 & 0\\
        0 & -w & -u^2 & 0\\
        0 & v & 0 & -u^2\\
        \hline    
        \multicolumn{2}{c|}{\textbf{0}_{3,2}} & \multicolumn{2}{c}{[u, v, w]^T[v,w]     }   
    \end{array}\right);
    \]  
    we denote the column span of this matrix \(\Lambda_p\).
    Thus the fiber of \(\phi\) above \(p\) is \(\Lambda_p\cap V(q_0)\), which is given by the rank \(3\) quadratic form associated to the \(4\times4\) symmetric matrix \(B_{p}^T A_{0} B_{p}\). Let \(\mathbf{M}\) denote the negative of the middle \(2\times 2\) submatrix of \(B_{p}^T A_{0} B_{p}\); \(\mathbf{M}\) has rank \(2\). One can check by inspection of \(B_{p}\) and \(A_0\) that \(B_{p}^T A_{0} B_{p}\) is block diagonal with the top left entry \(1\) and that the bottom right \(3\times 3\) block has rank \(2\).  Thus, the Brauer class of the generic fiber is given by the class of the rank \(2\) quadratic form associated to \(\mathbf{M}\).  One can also check that the top left entry of \(\mathbf{M}\) is \(Q_1\) and that the discriminant of the associated quadratic form, which is \(-\det(\mathbf{M})\), is equal to \(-(ab)^{-1}v^2(Q_2^2 - Q_1Q_3)\).  Therefore,
    \begin{align*}
     [(\Bl_C Z)_{\eta}] = (-(ab)^{-1}v^2(Q_2^2 - Q_1Q_3), Q_1) & = (-ab(Q_2^2 - Q_1Q_3), Q_1)\\& = (Q_2^2 - Q_1Q_3, Q_1) + (-ab, Q_1).
    \end{align*}
    Furthermore, the Brauer relations show \((-ab, av^2 + bw^2) = (-ab, a) = (b,a)\), as desired.
\end{proof}

%%%%%%%%%%%%%%%%%%%%%%%%%%%%%%%%%%%%%%%%%%%%%%%%%%%%%%%%%%%%%%%%%%%%%%%%%%%%%%%%%%%%%%%%%%%%%%%%
%%%%%%%%%%%%%%%%%%%%%%%%%%%%%%%%%%%%%%%%%%%%%%%%%%%%%%%%%%%%%%%%%%%%%%%%%%%%%%%%%%%%%%%%%%%%%%%%
\section{A constant Brauer class from \(k\)-points of \(\Ptildem{1}\)}\label{sec:alpha} 
%%%%%%%%%%%%%%%%%%%%%%%%%%%%%%%%%%%%%%%%%%%%%%%%%%%%%%%%%%%%%%%%%%%%%%%%%%%%%%%%%%%%%%%%%%%%%%%%
%%%%%%%%%%%%%%%%%%%%%%%%%%%%%%%%%%%%%%%%%%%%%%%%%%%%%%%%%%%%%%%%%%%%%%%%%%%%%%%%%%%%%%%%%%%%%%%%

Let \(\pi\colon X \to \PP^2\) be a geometrically standard conic bundle whose discriminant curve is a smooth quartic, as in Section~\ref{sec:IJTBackground}.
In this section, we pull back \(\pi\colon X \to \PP^2\) to obtain a conic bundle over an open subset of \(\Ptildem{1}\) and then show that if \(\Ptildem{1}(k)\neq\emptyset\) then this conic bundle is constant, i.e., its Brauer class is in the image of the map \(\Br k\to \Br \kk(\Ptildem{1})\). This construction can be viewed as a globalization of the argument in \cite{FJSVV}*{Section 6.4}.

By Equation~\eqref{def:Ptilde1}, a point \(x\in\Ptildem{1}\) corresponds to an effective divisor \(D_x\) on \(\Deltatilde_{\kk(x)}\), whose image under the pushforward by \(\varpi\colon\Deltatilde \to \Delta\) is contained in a line in \(\PP^2\) (defined over \(\kk(x)\)).  We write \(\ell(x) \colonequals \Span(\varpi_*{D_x})\) 
for this line.  Let \(I\) be the universal line over \(\Ptildem{1}\):
\[I \colonequals \{(p, x) \in \PP^2 \times \Ptildem{1}  : p \in \ell(x)\}.\]
This comes equipped with  projections \(p_1 \colon I \to \PP^2\) and \(p_2 \colon I \to \Ptildem{1}\). By definition, the fiber \(p_2^{-1}(x)\) is the line \(\ell(x)\), so \(p_2\) has the following structure.
\begin{lemma}
\label{lem:two-torsion-ptilde-i-x}
The second projection \(p_2 \colon I \to \Ptildem{1}\) is a \(\PP^1\)-bundle, and any line \(\ell \subset \PP^2\) gives a rational section \(I_\ell \colonequals p_1^{-1}(\ell) \subset I\).  Further, \(p_2^*\colon \Br \Ptildem{1}[2] \to \Br I[2]\) is an isomorphism. 
\end{lemma}
\begin{proof}
The first statement is immediate from the definition of \(I\). The definition also implies that \(p_2\) is smooth. Thus, since we assumed \(2\) is invertible in \(k\), we have an isomorphism on the \(2\)-primary torsion of \(\Br_{\mathrm{vert}}(I/\Ptildem{1})\) \cite{CTS-Brauer-book}*{Definition 11.1.1} and the pullback of \(\Br(\Ptildem{1})\) by \cite{CTS-Brauer-book}*{Proof of Corollary 11.1.6(ii)} (noting that \cite{CTS-Brauer-book}*{Equation (11.3)} holds for \(p\)-primary torsion when \(p\in k^{\times}\)). Finally, by the discussion following \cite{CTS-Brauer-book}*{Definition 11.3.1},  \(\Br_{\mathrm{vert}}(I/\Ptildem{1}) = \Br(I)\).
\end{proof}

Let \(x \in \Ptildem{1}(k)\)
and let \(U\colonequals \PP^2\smallsetminus \Delta\).
We have the following commutative diagrams of varieties and Brauer groups:
\begin{equation}\label{eq:diagram-alphatilde1}
    \begin{tikzcd}
    & I \arrow[ddr, "p_2"] \arrow[ddl, "p_1", swap]& \\
    & I_{\ell(x)} \arrow[u, hook] \arrow[d]  \arrow[dr, "\mathrm{bir}", swap, {anchor=south, rotate=-40}, leftrightarrow, dashed]& \\
    \PP^2 & \ell(x) \arrow[l, hook'] & \Ptildem{1}&
    \end{tikzcd}
    \begin{tikzcd}
    & \Br I_{U}[2] \arrow[d] & \Br I[2]\arrow[l, hook']\\
    & \Br \kk(I_{\ell(x)})[2] & \\
    \Br U[2] \arrow[uur, "p_1^*"] \arrow[r] & \Br \kk(\ell(x))[2] \arrow[u] & 
     \Br \Ptildem{1}[2] \arrow[uu,  "p_2^*" {swap}, "\sim" {anchor=south, rotate=90}]   \arrow[ul, hook']
    \end{tikzcd}
\end{equation}

\begin{thm}[Strengthening of \cite{FJSVV}*{Proposition~6.5}]\label{thm:alphatilde1}
The Brauer class \(p_1^*[X_U]\) is unramified on \(I\) and therefore is an element of \(\Br I[2]\). If there exists a point \(x\in\Ptildem{1}(k)\), then \(p_1^*[X_U]\) is constant and is equal to \([X_w]\) for any point \(w\in (\ell(x)\smallsetminus (\Delta\cap \ell(x)))(k)\).  Moreover \([X_w] = [X_{\kk(\ell(x))}]\in \Br \kk(\ell(x))\).
\end{thm}

\begin{defn}\label{def:alphatilde1}
    Assume that \(\Ptildem{1}(k)\neq \emptyset\). Let \(x\in \Ptildem{1}(k)\) and define \[\alphatildem{1} \colonequals x^*(p_2^*)^{-1}\left(p_1^*[X_{U}]\right)\in \Br k[2].\]  In particular, \(\alphatildem{1} = [X_w]\) for any \(w\in (\ell(x)\smallsetminus (\Delta\cap \ell(x)))(k)\).
\end{defn}

\begin{remark}
Theorem~\ref{thm:alphatilde1} implies that \(\alphatildem{1}\) is independent of the choice of \(x\in \Ptildem{1}(k)\). However, the definition of \(\alphatildem{1}\in \Br k\) \emph{does} require that \(\Ptildem{1}(k)\neq \emptyset\), which motivates our choice of notation.
\end{remark}

\begin{proof}[Proof of Theorem~\ref{thm:alphatilde1}]
    First we will assume that \(p_1^*[X_{U}]\) is unramified on \(I\), and, from this, deduce the rest of the Theorem. Then at the end, we will prove that \(p_1^*[X_U]\) is unramified.

    Since \(p_1^*[X_U]\) is unramified on \(I\), its restriction to the fiber \(p_2^{-1}(x)\) over \(x \in \Ptildem{1}(k)\) is unramified.  Since \(p_2^{-1}(x)\) maps isomorphically onto \(\ell(x)\) under \(p_1\), the conic bundle \(\pi|_{\ell(x)} \colon X_{\ell(x)} \to \ell(x)\) is unramified and so, since \(\ell(x) \simeq \PP^1\), \([X_{\kk(\ell(x))}] = [X_w]\).
    So the image of \([X_{\kk (\ell(x))}]\) under \(\Br\kk(\ell(x))[2]\to\Br\kk(I_{\ell(x)})[2]\) is the  \(2\)-torsion constant algebra \([X_w]\). This is equal to the image of \(p_1^*[X_U]\) in \(\Br\kk(I_{\ell(x)})\) by commutativity of the left triangle of the right-hand diagram in~\eqref{eq:diagram-alphatilde1}.
    Furthermore, the rightmost commutative square
    in~\eqref{eq:diagram-alphatilde1} shows that \(\Br I[2] \to \Br \kk(I_{\ell(x)})[2]\) is injective. Thus we deduce that \([X_w] = p_1^*[X_{U}]\).

    Now we prove the claim that \(p_1^*[X_U] \in \Br I_U\) is unramified, i.e., extends to an element of \(\Br I\).  By the purity exact sequence \cite{CTS-Brauer-book}*{Theorem 3.7.3} and functoriality, it remains to prove that \(p_1^*[X_{U}]\) is unramified at every codimension \(1\) point of \(I\) lying over \(\Delta\), i.e., every \(2\)-dimensional component of \(I_{\Delta}\), since $\dim I = 3$.\footnote{In fact, one can show that \(I_{\Delta}\) is pure of dimension \(2\), so the \(2\)-dimensional restriction is superfluous. However, we do not need the fact that \(I_{\Delta}\) is pure in the proof, so we omit this step.}
    By definition, a point of \(I_\Delta\) is a tuple \((p, x) \in \Delta \times \Ptildem{1}\) such that \(p\in \ell(x) = \Span(\varpi_*{D_x})\).  Since \(p\in \Delta\), this implies that \(p\in \varpi_*{D_x}\). If \(\varpi_*D_x\) has multiplicity \(1\) at \(p\), there is a unique point \(q_{p, D_x}\in D_x\) of \(\Deltatilde\) lying over \(p \in \Delta\).  Away from the \(1\)-dimensional locus of \(I_{\Delta}\) consisting of points \((p, x)\) where \(\varpi_*D_x\) has multiplicity at least \(2\) at \(p\),
     the map \((p, x) \mapsto q_{p, D_x}\) defines a rational map \(I_{\Delta}\dashrightarrow \Deltatilde\) that factors the projection \(I_{\Delta}\to \Delta\). In particular, every \(2\)-dimensional component of \(I_{\Delta}\) must have a nonconstant rational map to \(\Deltatilde\). Since the residue of \([X_{U}]\in \Br U\) along \(\Delta\) is given by \(\Deltatilde\), this implies that the residue of \(p_1^*[X_U]\) along each 2-dimensional component of \(I_\Delta\) is trivial, so the pullback \(p_1^*[X_U]\) lands in \(\Br I[2]\).
\end{proof}

%%%%%%%%%%%%%%%%%%%%%%%%%%%%%%%%%%%%%%%%%%%%%%
\section{Double covers of \(\PP^1\times \PP^2\) and obstructions to sections of \(\piQuadricSurface\)}\label{sec:obstruction-to-sections}

Any smooth double cover of
\(\PP^1_{[t_0:t_1]} \times \PP^2_{[u:v:w]}\) branched over a \((2,2)\)-divisor has equation
\begin{equation}\label{eq:DoubleCover}
z^2 = t_0^2 Q_1 + 2t_0t_1 Q_2 + t_1^2 Q_3,
\end{equation}
for a tuple of quadratic forms \(Q_1, Q_2, Q_3 \in k[u, v, w]\).  Further the quadratic forms are well-defined up to an action of \(\Aut(\PP^1)\), and this action agrees with the action of \(\PGL_2\) on the defining equations of \(\Deltatilde\) and \(\Delta\) described in Section~\ref{sec:etale_double_covers}, so the isomorphism class of the double cover depends only on \(\Deltatilde/\Delta\).  We write \(Y_{\Deltatilde/\Delta}\to \PP^1_{[t_0:t_1]} \times \PP^2_{[u:v:w]}\) for this isomorphism class.  If \(\Deltatilde\) and \(\Delta\) are smooth, then so is \(Y_{\Deltatilde/\Delta}\)~\cite{FJSVV}*{Theorem 6.1(a)}.

The second projection \(\pi\colon Y_{\Deltatilde/\Delta} \to \PP^2\) gives \(Y=Y_{\Deltatilde/\Delta}\) the structure of a standard conic bundle with discriminant cover \(\Deltatilde/\Delta\), and the first projection
\(\piQuadricSurface \colon Y\to \PP^1\) gives \(Y\) the structure of a quadric surface bundle.
By \cite{FJSVV}*{Theorem 4.5(ii) and Proposition 6.3(iii)}, the Prym curve \(\Gamma = \Gamma_{\Deltatilde/\Delta}\) arises from the Stein factorization of \(\calF_1(Y/\PP^1) \to \PP^1\), so \([\calF_1(Y/\PP^1)]\) gives a Brauer class on \(\Gamma\).  Furthermore, \(\piQuadricSurface\) has a section if and only if \([\calF_1(Y/\PP^1)] = 0\in \Br \Gamma\).\footnote{Any quadric surface has a point if and only if it has lines defined over the discriminant extension.  Indeed, if the quadric surface has a point, then slicing the surface with the tangent plane at that point will give a geometrically reducible conic whose geometric components are defined over the discriminant extension.  Conversely, if there are lines defined over the discriminant extension, then the intersection of a line and its Galois conjugate will give a rational point.\label{fn:QuadricSurfaceSection}}

\begin{thm}
\label{thm:brauer-class-fano-variety-of-lines}
Let \(k\) be a field of characteristic different from \(2\), let \(\varpi\colon \Deltatilde\to \Delta\) be a geometrically integral \'etale double cover of a smooth plane quartic over \(k\), and let \(Y= Y_{\Deltatilde/\Delta}\to \PP^1\times \PP^2\) be the double cover associated to \(\Deltatilde/\Delta\). Then for any \(x\in \Ptildem{1}(k)\) and any \(w\in (\ell(x)\smallsetminus (\ell(x)\cap \Delta))(k)\) , 
 \[
    [\calF_1(Y/\PP^1)] = [Y_w]_{\Gamma} = (\alphatildem{1})_{\Gamma} \in \Br \Gamma.
 \]
\end{thm}
We prove this in Section~\ref{sec:ProofOfBrClassThm}. Before doing so, we record the following corollaries.

\begin{cor}\label{cor:brauer-class-fano-variety-of-lines}
Let \(\Deltatilde\),\(\Delta\) and \(Y= Y_{\Deltatilde/\Delta}\) be as in Theorem~\ref{thm:brauer-class-fano-variety-of-lines}. Consider the following conditions:
\begin{enumerate}
    \item \(\Ptildem{1}(k)\neq\emptyset\) and \(\tilde{\alpha}^{(1)}\in \ker (\Br k\to \Br \Gamma)\) \textup{(}i.e., \((\tilde{\alpha}^{(1)})_{\Gamma} = 0\in \Br \Gamma\)\textup{)}.
    \label{it:IJTandBrVanish}
    \item \(\piQuadricSurface\colon Y \to \PP^1\) has a section (over \(k\)).\label{it:pi1section}
    \item \(Y\) is rational (over \(k\)).\label{it:ratl}
\end{enumerate}
Then \eqref{it:IJTandBrVanish}\(\Rightarrow\)\eqref{it:pi1section}\(\Rightarrow\)\eqref{it:ratl}.
\end{cor}

\begin{proof}
    \eqref{it:IJTandBrVanish}\(\Rightarrow\)\eqref{it:pi1section}:
    By Theorem~\ref{thm:brauer-class-fano-variety-of-lines}, \((\alphatildem{1})_{\Gamma}=[\calF_1(Y/\PP^1)]\in \Br \Gamma\),
    and hence \([\calF_1(Y/\PP^1)] = 0\in \Br \Gamma\).
    Therefore \(\piQuadricSurface\) has a section (see  Footnote~\ref{fn:QuadricSurfaceSection}).

\eqref{it:pi1section}\(\Rightarrow\)\eqref{it:ratl}: The generic fiber of \(\piQuadricSurface\) is a quadric, which is rational as soon as it has a point.
\end{proof}

\begin{remark}
By Proposition~\ref{prop:IJTforConicBundles}, the vanishing of the IJT obstruction for \(Y\) is equivalent to either \(\Ptildem{1}(k)\neq\emptyset\) or \(\Ptilde(k)\neq\emptyset\). Thus, the condition~\eqref{it:IJTandBrVanish} in Corollary~\ref{cor:brauer-class-fano-variety-of-lines} is stronger than assuming only that the IJT obstruction vanishes.
\end{remark}

\begin{cor}\label{cor:localfield_index2}
    Let \(k\) be a local field of characteristic different from \(2\), let \(\Deltatilde,\Delta\) be as above, and let \(Y= Y_{\Deltatilde/\Delta}\to \PP^1\times \PP^2\) be the double cover associated to \(\Deltatilde/\Delta\).  If \(\Gamma\) has index \(2\), i.e., if \(\Gamma(L) = \emptyset\) for all extensions \(L/k\) of odd degree, then 
    \(Y\) is rational if and only if the IJT obstruction vanishes.
\end{cor}
\begin{proof}
We will first show that the vanishing of the IJT obstruction is equivalent to~\eqref{it:IJTandBrVanish} in Corollary~\ref{cor:brauer-class-fano-variety-of-lines}. First, by \cite{Lichtenbaum}*{Theorem~7} and \cite{poonen-stoll}*{Footnote~10}, since \(k\) is a local field we have \(\bPic^1_\Gamma(k)\neq\emptyset\) (see Remark~\ref{rem:Real}\ref{it:Pic1} for \(k=\RR\)). So by Proposition~\ref{prop:IJTforConicBundles}, the vanishing of the IJT obstruction is equivalent to \(\Ptildem{1}(k)\neq\emptyset\), and in this setting we may define \(\alphatildem{1}\in \Br k[2]\) (see Definition~\ref{def:alphatilde1}).
Furthermore, \(\Pic^1(\Gamma) = \emptyset\) since \(\Gamma\) has index \(2\), so the map \(\Pic(\Gamma) \to \bPic_{\Gamma}(k)\) has a nontrivial element of order \(2\) in its cokernel.  Since \(\Br k[2] \simeq \Z/2\Z\), the low degree terms of the Hochschild--Serre spectral sequence imply that \(\Br k[2] \to \Br \Gamma[2]\) is identically \(0\). In particular, the condition \(\alphatildem{1}\in \ker(\Br k\to \Br \Gamma)\) always holds. This shows that~\eqref{it:IJTandBrVanish} holds if and only if the IJT obstruction vanishes.

Corollary~\ref{cor:brauer-class-fano-variety-of-lines} shows \eqref{it:IJTandBrVanish}\(\Rightarrow\)\eqref{it:ratl}. Conversely, if~\eqref{it:ratl} holds, then the IJT obstruction vanishes by~\cite{bw-ij}*{Theorem 3.11}.
\end{proof}

\subsection{Proof of Theorem~\ref{thm:brauer-class-fano-variety-of-lines}}\label{sec:ProofOfBrClassThm}

    The second equality holds by Definition~\ref{def:alphatilde1}. For the first equality, it suffices to show equality in \(\Br \kk(\Gamma)\). By \cite{Amitsur}*{Corollary~9.5}, if two finite-dimensional central simple algebras over \(\kk(\Gamma)\) have the same splitting fields over \(\kk(\Gamma)\), then they generate the same subgroup in \(\Br \kk(\Gamma)\). Since  \([Y_w]\) and  \([\calF_1(Y/\PP^1)]\) are both \(2\)-torsion, showing they have the same splitting fields will imply the desired equality. 

Let \(k'\) be an extension of \(\kk(\Gamma)\) which is a splitting field for \([\calF_1(Y/\PP^1)]\). Then the generic fiber of \(Y\to \PP^1\) contains lines over \(k'\), so, over \(k'\), the generic fiber becomes isomorphic to \(\PP^1_{k'}\times \PP^1_{k'}\).  In particular, every hyperplane section of the generic fiber contains \(k'\)-points. Note that \((Y_{\ell(x)})_{k'}\) is a hyperplane section of the generic fiber, so then \(Y_{\ell(x)}(k')\neq \emptyset\). By Theorem~\ref{thm:alphatilde1}, \(Y_{\ell(x)}\) is Brauer equivalent to \(Y_w\), so \(Y_w(k')\neq\emptyset.\)

Now assume that \(k'\) is an extension of \(\kk(\Gamma)\) which is a splitting field for \(Y_w\).  Then since \(Y_w\subset Y_{\ell(x)}\), we have \(Y_{\ell(x)}(k')\neq\emptyset\), so by Theorem~\ref{thm:alphatilde1} \(Y_{\ell(x)}\to \ell(x)\) has a section defined over \(k'\). Furthermore, the point \(x\in \Ptildem{1}(k)\) by definition (see~\eqref{def:Ptilde1}) gives a maximal collection \(W\) of pairwise disjoint geometric \((-1)\)-curves of \(Y_{\ell(x)}\) that is defined over \(k'\). The complementary set of geometric \((-1)\)-curves \(W'\) is then also defined over \(k'\). By applying Lemma~\ref{lem:ConicBundle} below to the blow-down at \(W'\), we obtain a \(k'\) section \(\ell(x)\to Y_{\ell(x)}\) that meets every geometric component of \(W\) and is disjoint from \(W'\). By~\cite{FJSVV}*{Proposition 5.10(iii) and Theorem 6.4 (cf. Proof of Proposition 6.5)}, the image of this section is a curve that maps with odd degree to its image under \(\pi_1\colon Y \to \PP^1\).  By Springer's theorem~\cite{Springer52}, \(Y \to \PP^1\) then has a section over \(k'\) so \([\calF_1(Y/\PP^1)]\) is trivial in \(\Br k'\).\qed

\begin{lemma}\label{lem:ConicBundle}
    Let \(\pi \colon X \to \PP^1\) be a smooth morphism such that every fiber is a smooth conic and \(X(k) \neq \emptyset\).  For any finite Galois-invariant set \(S \subset X(\kbar)\) containing at most one point in each geometric fiber, there exists a section of \(X\) defined over \(k\) avoiding \(S\).
\end{lemma}
\begin{proof}
    The assumptions imply that \(X\) is a trivial conic bundle with a rational point, and hence is isomorphic to a Hirzebruch surface.  Let \(\sigma\in\Pic X\) denote the class of a section of \(\pi\), and let \(f\) denote the class of a fiber.  
    Let \(n \colonequals \#S\), which by assumption is also the number of geometric fibers on which \(S\) is supported.  Let \(F\) denote the union of geometric fibers containing \(S\); the divisor \(F\) has class \(nf\).  For \(b \gg 0\), by ampleness of \(\OO_{\PP^1}(1)\), we have
    \[H^1(\PP^1, \pi_* \OO_X(\sigma + (b-n)f)) = H^1(\PP^1, \pi_*\OO_X(\sigma) \otimes \OO_{\PP^1}(b-n)) = 0.\]
    Furthermore, since \(\sigma + (b-n)f\) restricts to \(\OO_{\PP^1}(1)\) on the fibers of \(\pi\), by the theorem on cohomology and base change, we have, \(H^0(\PP^1, R^1\pi_*\OO_X(\sigma + (b-n)f)) = 0\).
    Thus, by the Leray spectral sequence, we have \(H^1(X, \OO_X(\sigma + (b-n)f)) = 0\).  Thus, for \(b \gg 0\), the restriction map
    \[H^0(X, \OO_X(\sigma + bf)) \to H^0(F, \OO_X(\sigma + bf)|_F)\]
    is surjective.  On each geometric irreducible component of \(F\), the class \(\sigma + bf\) restricts to \(\OO_{\PP^1}(1)\), which is basepoint free.  Hence, there exists a section in \(H^0(F, \OO_X(\sigma + bf)|_F)\) avoiding \(S\), which by the above surjectivity can be lifted to a section on \(X\).
\end{proof}

%%%%%%%%%%%%%%%%%%%%%%%%%%%%%%%%%%%%%%%%%%%%%%%%%%%%%%%%%%%%%%%%%%%%%%%%%%%%%%%%%%%%%%%%%%%
%%%%%%%%%%%%%%%%%%%%%%%%%%%%%%%%%%%%%%%%%%%%%%%%%%%%%%%%%%%%%%%%%%%%%%%%%%%%%%%%%%%%%%%%%%%
\section{Proof of Theorem~\ref{thm:intro-mainthm-1}}\label{sec:proof}
%%%%%%%%%%%%%%%%%%%%%%%%%%%%%%%%%%%%%%%%%%%%%%%%%%%%%%%%%%%%%%%%%%%%%%%%%%%%%%%%%%%%%%%%%%%
%%%%%%%%%%%%%%%%%%%%%%%%%%%%%%%%%%%%%%%%%%%%%%%%%%%%%%%%%%%%%%%%%%%%%%%%%%%%%%%%%%%%%%%%%%%

Let \(Q_1, Q_2, Q_3 \in k[u, v, w]\) be homogeneous forms defining a smooth plane quartic \(\Delta\) and a geometrically integral \'etale double cover \(\Deltatilde\to \Delta\) as in \eqref{eq:Delta}--\eqref{eq:Deltatilde}.
By assumption, there exists \(y\in \Gamma_{\Deltatilde/\Delta}(k)\) which, without loss of generality, we assume  lies over \([1:0]\in \PP^1(k)\) so that \(Q_1\) either is rank \(2\) or has square discriminant. 
By Proposition~\ref{prop:construction-intersection-of-quadrics}, there exists a smooth complete intersection of two quadrics \(Z\subset \PP^5\) containing a conic \(C\) such that \(\Bl_C Z\to \PP^2\) has discriminant cover \(\Deltatilde/\Delta\).  Furthermore, the generic fiber of the conic bundle has Brauer class
\[
    [(\Bl_C Z)_{\eta}] = (Q_1, Q_2^2 - Q_1Q_3) + [Q_1] \in \Br \kk(\PP^2).
\]

Since \(Y\) is defined by the equation \(z^2 = Q_1(u,v,w)t_0^2 + 2Q_2(u,v,w)t_0t_1 + Q_3(u,v,w)t_1^2\), completing the square shows that the Brauer class of the generic fiber of the conic bundle \(\pi\colon Y \to \PP^2\) is equal to the quaternion algebra \((Q_1, Q_2^2 - Q_1Q_3)\).  In particular, 
\[
    [(\Bl_C Z)_{\eta}] - [Y_\eta] = [Q_1]\in \Br k[2].
\]

On the other hand, the fiber of \(\piQuadricSurface\colon Y \to \PP^1\) above \([1:0]\), described in \eqref{eq:DoubleCover}, is given by \(z^2 = Q_1\), so \([Q_1]\) is equal to the evaluation of \([\calF_1(Y/\PP^1)]\) at \(y\). By Theorem~\ref{thm:brauer-class-fano-variety-of-lines}, \(\calF_1(Y/\PP^1)\) is a constant algebra, so it is equal to its evaluation at \(y\).  Furthermore, the class of the relative Fano variety of lines \([\calF_1(Y/\PP^1)]\) is exactly the obstruction to \(\piQuadricSurface\) having a section (see Footnote~\ref{fn:QuadricSurfaceSection}). Finally, we deduce that \(\Bl_C Z\) is rational by Corollary~\ref{cor:intro-inline}.\qed

%%%%%%%%%%%%%%%%%%%%%%%%%%%%%%%%%%%%%%%%%%%%%%%%%%%%%%%%%%%%%%%%%%%%%%%%%
\section{Real topological consequences of Theorem~\ref{thm:intro-mainthm-1}}\label{sec:real}
%%%%%%%%%%%%%%%%%%%%%%%%%%%%%%%%%%%%%%%%%%%%%%%%%%%%%%%%%%%%%%%%%%%%%%%%%

The goal of this section is to show that real topological obstructions to rationality imply that for a fixed discriminant cover \(\Deltatilde/\Delta\) (with no restrictions on \(\deg \Delta\)), there is a unique conic bundle structure on \(\mathbb{P}^3_\RR\) with discriminant cover \(\Deltatilde/\Delta\), with the possible exception of one isotopy class of \(\Delta(\RR)\). Precisely, we prove the following.
\begin{thm}\label{thm:MultipleConicBundleStrOverR}
    Let \(\pi\colon X \to \PP^2_\RR\) and \(\pi'\colon X'\to \PP^2_\RR\) be conic bundles with the same discriminant cover \(\Deltatilde/\Delta\) of smooth geometrically integral curves. Assume that the real points of \(X\) and \(X'\) are
    both nonempty and connected. Then either \(X_{\eta}\simeq X'_{\eta}\), or \(\Delta(\mathbb{R})\) is an \defi{oval} (i.e., a homotopically trivial simple closed curve in \(\PP^2(\RR)\)) and \(\Deltatilde(\mathbb{R}) =\emptyset\).
\end{thm}
\begin{cor}
        Let \(\pi\colon X \to \PP^2_\RR\) and \(\pi'\colon X'\to \PP^2_\RR\) be conic bundles with the same discriminant cover \(\Deltatilde/\Delta\) of smooth geometrically integral curves. If \(X\) and \(X'\) are both rational, then either \(X_{\eta}\simeq X'_{\eta}\), or \(\Delta(\mathbb{R})\) is an oval
        and \(\Deltatilde(\mathbb{R}) =\emptyset\).
\end{cor}
\begin{proof}
    Since \(X, X'\) are both rational, the loci of real points \(X(\RR), X'(\RR)\) are nonempty and connected \cite{DelfsKnebusch}*{Theorem~13.3}, so the statement follows from Theorem~\ref{thm:MultipleConicBundleStrOverR}.
\end{proof}

In particular, in the setting of Theorem~\ref{thm:MultipleConicBundleStrOverR}, if \(X_\eta\not\simeq X'_\eta\), then the classification of the real loci of smooth plane curves implies that \(\deg\Delta\) is necessarily even (see Section~\ref{sec:RealArbitraryDegree}).
In the degree \(4\) conic bundle case, from Theorem~\ref{thm:MultipleConicBundleStrOverR} and our earlier results (including Theorem~\ref{thm:intro-mainthm-1}), we deduce rationality criteria 
for all but one isotopy class of \(\Delta(\RR)\).

\begin{thm}\label{thm:real}
Let \(\varpi\colon \Deltatilde\to \Delta\) be a geometrically integral \'etale double cover of a smooth plane quartic over \(\RR\), and let \(Y_{\Deltatilde/\Delta}\) be the double cover defined in~\eqref{eq:DoubleCover}. Consider the following statements.
\begin{enumerate}
    \item The IJT obstruction vanishes and \(\tilde{\alpha}^{(1)}\in \ker (\Br \RR\to \Br \Gamma)\).\label{it:IJTandBrVanish-R}
    \item \(\piQuadricSurface\colon Y \to \PP^1\) has a section.\label{it:pi1section-R}
    \item \(Y\) is rational.\label{it:ratl-R}
    \item The IJT obstruction vanishes and \(Y(\RR)\) is nonempty and connected.\label{it:IJTandConnected}
\end{enumerate}
Then \eqref{it:IJTandBrVanish-R}\(\Rightarrow\)\eqref{it:pi1section}\(\Rightarrow\)\eqref{it:ratl-R}\(\Rightarrow\)\eqref{it:IJTandConnected}. If  \(\Delta(\RR)\) is not a single oval or if \(\Gamma(\RR) = \emptyset\), then \eqref{it:IJTandConnected}\(\Rightarrow\)\eqref{it:IJTandBrVanish-R}.  
\end{thm}

\begin{proof}
\eqref{it:IJTandBrVanish-R}\(\Rightarrow\)\eqref{it:pi1section-R}\(\Rightarrow\)\eqref{it:ratl-R} by Corollary~\ref{cor:brauer-class-fano-variety-of-lines} and Remark~\ref{rem:Real}\ref{it:Pic1}. The implication \eqref{it:ratl-R}\(\Rightarrow\)\eqref{it:IJTandConnected} follows because 
 the number of connected components of \(Y(\RR)\) is a birational invariant for smooth projective varieties~ \cite{DelfsKnebusch}*{Theorem~13.3} and the IJT obstruction is an obstruction to rationality (see Section~\ref{sec:IJTBackground}). It remains to prove \eqref{it:IJTandConnected}\(\Rightarrow\)\eqref{it:IJTandBrVanish-R} under the given assumptions.  
If \(\Gamma(\RR) =\emptyset\), then \(\Br \RR \to \Br \Gamma\) is identically \(0\) and so the implication is trivial.
 Now assume that \(\Gamma(\RR)\neq\emptyset\) and that \(\Delta(\RR)\) is not a single oval. In this case, \(\Br \RR \to \Br \Gamma\) is injective (evaluating at a point of \(\Gamma(\RR)\) gives a retraction), so we may freely identify constant Brauer classes with their images in \(\Br \Gamma\) (or \(\Br \kk(\PP^2)\)). By Theorems~\ref{thm:intro-mainthm-1},~\ref{thm:alphatilde1} and~\ref{thm:brauer-class-fano-variety-of-lines}, there is a rational threefold \(X'\) and a standard conic bundle \(\psi\colon X' \to \PP^2\) such that \([Y_{\eta}] - [X'_{\eta}] = \tilde{\alpha}^{(1)}\). Since \(X'\) is rational, its real points are nonempy and connected~\cite{DelfsKnebusch}*{Theorem~13.3}. Then the assumption that \(\Delta(\RR)\) is not a single oval together with Theorem~\ref{thm:MultipleConicBundleStrOverR} completes the proof.
\end{proof}
\begin{remarks}\label{rem:Real}\hfill
\begin{enumerate}[(i)]
    \item If the IJT obstruction vanishes, then the class \(\tilde{\alpha}^{(1)}\) is defined. Indeed, \(\bPic^1_{\Gamma}\simeq \bPic^3_{\Gamma}\) has an \(\RR\)-point because there exists a \(\Gal(\C/\RR)\)-invariant partition of the Weierstrass points of \(\Gamma\) into two sets of cardinality \(3\). Precisely, either there are at least two real Weierstass points, and thus the Weierstrass points can be divided into Galois-invariant subsets of order \(3\), or the Weierstrass points all come in complex conjugate pairs and \(\{\{\omega_1,\omega_2,\omega_3\},\{\overline{\omega}_1,\overline{\omega}_2,\overline{\omega}_3\}\}\) is the desired Galois-invariant partition. Then Proposition~\ref{prop:IJTforConicBundles} shows that the vanishing of the IJT obstruction is equivalent to \(\Ptildem{1}(\RR)\neq\emptyset\).\label{it:Pic1}
    \item In the case that \(\Delta(\RR)\) is not two nested ovals (nor a single oval), Theorem~\ref{thm:real} recovers an earlier result of M. Ji and the second author, namely the equivalences \eqref{it:pi1section}\(\Leftrightarrow\)\eqref{it:ratl}\(\Leftrightarrow\)\eqref{it:IJTandConnected}~\cite{JJ23}*{Theorem 1.2}. 
    \item The condition in~\eqref{it:IJTandBrVanish-R} cannot be replaced with the stronger condition that the IJT obstruction vanishes and \(\tilde{\alpha}^{(1)} = 0\in \Br \RR\). Indeed, there are examples where \(Y\) is rational over \(\RR\) and \(\tilde{\alpha}^{(1)}\in \Br \RR\) is the nontrivial element; see \cite{JJ23}*{Example 4.10}. If \(\Delta(\RR)\) is not a single oval, such examples necessarily have that \(\Br \RR\to \Br \Gamma\) is identically zero, which is equivalent to \(\Gamma(\RR)=\emptyset\).\label{it:rem-Real-alphatilde1-nonzero}
    \item When \(\Delta(\mathbb R)\) is a single oval and \(\Gamma(\RR)\neq\emptyset\), then the implication \eqref{it:IJTandConnected}\(\Rightarrow\)\eqref{it:IJTandBrVanish-R} can fail. Indeed, the implication \eqref{it:IJTandConnected}\(\Rightarrow\)\eqref{it:pi1section-R} can fail; a counterexample to this implication is provided by \cite{FJSVV}*{Example 1.6}. It remains an open question whether~\eqref{it:ratl-R}\(\Rightarrow\)\eqref{it:pi1section-R} or~\eqref{it:IJTandConnected}\(\Rightarrow\)\eqref{it:ratl-R} in general.\label{it:one-oval-R-unknown-example}
    \item The isotopy class of \(\Delta(\RR)\) does not in general determine whether \(\alphatildem{1}\) is in the kernel of \(\Br\RR\to\Br\Gamma\) (even when combined with the data of the topological type of \(\Deltatilde(\RR)\)). Indeed, there are examples where the condition that \(\tilde{\alpha}^{(1)}\in \ker (\Br \RR\to \Br \Gamma)\) differs when \(\Deltatilde(\RR)=\emptyset\) and \(\Delta(\mathbb R)\) is either empty \cite{JJ23}*{Examples 4.1 and 4.8}, two non-nested ovals (\cite{FJSVV}*{Theorem 1.4(i)} and \cite{JJ23}*{Example 4.11}), or three ovals \cite{JJ23}*{Examples 4.6 and 4.12}.
    \end{enumerate}
\end{remarks}

In Section~\ref{sec:RealArbitraryDegree}, we study the real topological properties of \(\pi(X(\mathbb{R}))\) for \(\pi\colon X \to \PP^2\) a conic bundle with \(\deg\Delta\) arbitrary. We then use these results to prove Theorem~\ref{thm:MultipleConicBundleStrOverR}. In Section~\ref{sec:RealDegree4} we specialize to degree \(4\) conic bundles and further explore the conditions in Theorem~\ref{thm:real} from a topological perspective.

%%%%%%%%%%%%%%%%%%%%%%%%%%%%%%%%%%%%%%%%%%%%%%%%%%%%%%%%%%%%%%%%%%%%%%%%%%%%%%%%%%%%%%%%%%%%
\subsection{Real points on conic bundles}\label{sec:RealArbitraryDegree}
%%%%%%%%%%%%%%%%%%%%%%%%%%%%%%%%%%%%%%%%%%%%%%%%%%%%%%%%%%%%%%%%%%%%%%%%%%%%%%%%%%%%%%%%%%%%

First we recall some facts about the topology of smooth real plane curves; see \cite{Mangolte-real}*{Section 2.7} for more details.
If \(\Delta\) is a smooth plane curve over \(\RR\), then its real locus \(\Delta(\RR)\) is a disjoint union of simple closed curves in \(\PP^2(\RR)\). If \(\deg\Delta\) is even, then \(\Delta(\RR)\) is a (possibly empty) disjoint union of ovals (recall an oval is a homotopically trivial simple closed curve in \(\PP^2(\RR)\)). Furthermore, in the even degree case, the complement \(\PP^2(\RR)\setminus\Delta(\RR)\) has a unique non-orientable component, called the \defi{outside} of \(\Delta(\RR)\). If \(\deg\Delta\) is odd, then exactly one connected component of \(\Delta(\RR)\) is a \defi{pseudo-line} (i.e., a simple closed curve in \(\PP^2(\RR)\) that is not homotopically trivial) and the other connected components are ovals. The complement of a pseudo-line in \(\mathbb P^2(\RR)\) is connected and is homeomorphic to a disc.
\begin{prop}\label{prop:ImageOfRealPoints}
    Let \(\varpi\colon\Deltatilde\to\Delta\) be a geometrically integral \'etale double cover of a smooth plane curve over \(\RR\), and let \(\pi\colon X\to \PP^2_{\RR}\) be a conic bundle with discriminant cover \(\Deltatilde/\Delta\).
    \begin{enumerate}
        \item The boundary \(\partial(\pi(X(\RR)))\) is equal to \(\Delta(\RR)\smallsetminus \varpi(\Deltatilde(\RR))\).\label{it:boundary}
        \item If \(\deg\Delta\) is odd then \(\varpi(\Deltatilde(\RR))\) contains the pseudo-line of \(\Delta(\RR)\); in particular, \(\Deltatilde(\RR)\neq\emptyset\).\label{it:odd-deg-pseudo-line-in-im}
        \item The sets \(\pi(X(\RR))\) and \((\PP^2(\RR)\smallsetminus \pi(X(\RR)))\) are both nonempty and connected if and only if \(\Delta(\RR)\smallsetminus \varpi(\Deltatilde(\RR))\) is an oval. \label{item:connected-image-complement-1}
        \item The sets \(\pi(X(\RR))\) and \((\PP^2(\RR)\smallsetminus \pi(X(\RR))) \cup \Delta(\RR)\) are both nonempty, connected, and 2-dimensional if and only if \(\Delta(\RR)\) is an oval and \(\Deltatilde(\RR)=\emptyset\). \label{item:connected-image-complement-2}
    \end{enumerate}
\end{prop}
Before proving this proposition we show how it implies Theorem~\ref{thm:MultipleConicBundleStrOverR}.

\begin{proof}[Proof of Theorem~\ref{thm:MultipleConicBundleStrOverR}]
Assume that \(X_\eta \not\simeq X_\eta'\).
For any \(w\in \PP^2(\RR)\smallsetminus \Delta(\RR)\), since \([X_w] - [X'_w] = [X_{\eta}] - [X'_\eta]\) is nontrivial, \(X_w(\RR) \neq \emptyset\) if and only if \(X'_w(\RR) = \emptyset\). Further, \(X_{\Delta}(\RR) \to \Delta(\RR)\) is surjective for any conic bundle with discriminant \(\Delta\). Thus \(\pi'(X'(\RR)) = (\PP^2(\RR)\smallsetminus \pi(X(\RR))) \cup \Delta(\RR)\).
Since \(X(\RR),X'(\RR)\) are connected 3-manifolds and \(\pi,\pi'\) are conic bundles, the images \(\pi(X(\RR)),\pi'(X'(\RR))\) must be connected and 2-dimensional, so~Proposition~\ref{prop:ImageOfRealPoints}\eqref{item:connected-image-complement-2} implies the result.
    \end{proof}

\begin{proof}[Proof of Proposition~\ref{prop:ImageOfRealPoints}]\hfill

    \textbf{\eqref{it:boundary}:} Fix a Zariski affine open \(U = \Spec A\subset \PP^2\) where \(X_{U}\) can be viewed as a subvariety of \(\PP^2_A\), and fix a defining quadratic equation \(F\) for \(X_U\). Since \(\pi\) is smooth away from \(\Delta\), the signature of \(F\) is constant on real connected components of \(U(\RR)\smallsetminus (\Delta\cap U)(\RR)\).  Note that a connected component \(B\) of \(U(\RR)\smallsetminus (\Delta\cap U)(\RR)\) is in the image of \(X_U(\RR)\) if and only if \(F|_B\) is an indefinite quadratic form.
    
    Fix a real connected component \(c\) of \((\Delta\cap U)(\RR)\) and let \(B_1,B_2\) be the two connected components of \(U(\RR)\smallsetminus (\Delta\cap U)(\RR)\) such that \(B_1\cup B_2\cup c\) becomes connected, i.e., \(B_1,B_2\) are the immediate interior and immediate exterior of \(c\). The statement now follows from the following two claims.
    \begin{enumerate}
        \item The signatures of \(F|_{B_1}\) and \(F|_{B_2}\) differ by exactly one sign. In particular, we may assume that \(F|_{B_1}\) is indefinite.
        \item The quadratic form \(F|_{B_2}\) is also indefinite if and only if \(c\subset \varpi(\Deltatilde(\RR))\).
    \end{enumerate}

    By continuity, it suffices to prove both claims in a (Euclidean) open neighborhood around a point \(x_0\in c \subset(\Delta\cap U)(\RR)\).  In an open neighborhood of \(x_0\), we may change coordinates so that \(F\) is diagonal and two of the coefficients are nowhere vanishing (see~\cite{IOOV}*{Lemma 6.3}).  Thus, the sign of those coefficients are equal on \(B_1\) and \(B_2\).  Since the vanishing of the discriminant of \(F\) exactly agrees with \(\Delta\), the third coefficient must vanish with multiplicity \(1\) exactly along \(\Delta\).  In particular, it has a different sign on \(B_1\) than on \(B_2\).  This completes the proof of the first claim.

    To prove the second claim, we assume without loss of generality that the signature of \(F|_{B_1}\) is \((+,+,-)\).  By the first claim, the signature of \(F|_{B_2}\) is either \((+,+,+)\) or \((+,-,-)\).  Furthermore, the proof of the first claim above shows that the signature of \(F|_{B_2}\) is \((+,-,-)\) if and only if the signature of \(F|_{x_0}\) is \((+,0,-)\).  Finally, since  \(\Delta\) is smooth and \(\Deltatilde\) is geometrically integral,
    \(F|_{x_0}\) is a rank \(2\) quadratic form, and its vanishing has a smooth \(\RR\)-point if and only if \(x_0\in\varpi(\Deltatilde(\RR))\). This completes the proof of the second claim.

    \textbf{\eqref{it:odd-deg-pseudo-line-in-im}:}    
    Let \(c\subset\Delta(\RR)\) be the pseudo-line. Let \(B\) be the  unique connected component of \(\PP^2(\RR)\setminus\Delta(\RR)\) whose closure contains \(c\); for any point \(x_0\in c\) there is a (Euclidean) open neighborhood of \(x_0\) contained in \(c\cup B\). As argued in the proof of part~\eqref{it:boundary}, \(B\) is either contained in or disjoint from \(\pi(X(\RR))\).   
    Since \(X(\RR)\) is a 3-manifold, none of the connected components of its image are 1-manifolds. So, since \(c\subset\Delta(\RR)\subset\pi(X(\RR))\), we must have \(B\subset\pi(X(\RR))\).
    This shows that \(c\) is in the interior of \(\pi(X(\RR))\), and therefore~\eqref{it:boundary} implies \(c\subset\varpi(\Deltatilde(\RR))\).
    
\textbf{\eqref{item:connected-image-complement-1}:} 
        The loci \(\pi(X(\RR))\) and \(\PP^2(\RR)\smallsetminus \pi(X(\RR))\) are both nonempty and connected if and only if the boundary \(\partial(\pi(X(\RR)))\) is connected.
        Thus,~\eqref{item:connected-image-complement-1} follows from~\eqref{it:boundary}, the Jordan curve theorem, the classification of the real loci of smooth plane curves, and~\eqref{it:odd-deg-pseudo-line-in-im}.
    
    \textbf{\eqref{item:connected-image-complement-2}:} By~\eqref{it:boundary} we have \((\PP^2(\RR)\smallsetminus\pi(X(\RR)))\cup\Delta(\RR)=\overline{\PP^2(\RR)\smallsetminus \pi(X(\RR))} \sqcup \varpi(\Deltatilde(\RR))\), so this set is connected if and only if exactly one of \(\overline{\PP^2(\RR)\smallsetminus \pi(X(\RR))}\) or \(\varpi(\Deltatilde(\RR))\) is empty and the other is nonempty and connected.
    First, if \(\varpi(\Deltatilde(\RR))=\emptyset\) and \(\overline{\PP^2(\RR)\smallsetminus \pi(X(\RR))}\) is nonempty and connected, then \(\Delta(\RR)=\Delta(\RR)\smallsetminus\varpi(\Deltatilde(\RR))\). Since \(\overline{\PP^2(\RR)\smallsetminus \pi(X(\RR))}\) is nonempty and connected if and only if \(\PP^2(\RR)\smallsetminus \pi(X(\RR))\) is, the statement that \(\pi(X(\RR))\) and \((\PP^2(\RR)\smallsetminus \pi(X(\RR))) \cup \Delta(\RR)\) are both nonempty and connected if and only if \(\Delta(\RR)\) is an oval follows from~\eqref{item:connected-image-complement-1}.
    Furthermore, in this case exactly one of \(\pi(X(\RR))\) or \((\PP^2(\RR)\smallsetminus \pi(X(\RR))) \cup \Delta(\RR)\) is the closed disc enclosed by \(\Delta(\RR)\) and the other is a M\"obius band, so both are 2-dimensional.
    On the other hand, if \(\overline{\PP^2(\RR)\smallsetminus \pi(X(\RR))}=\emptyset\) and \(\varpi(\Deltatilde(\RR))\) is nonempty and connected, then \(\Delta(\RR) =(\PP^2(\RR)\smallsetminus\pi(X(\RR)))\cup\Delta(\RR)\) is 1-dimensional, so the left-hand condition of~\eqref{item:connected-image-complement-2} does not hold. Since \(\varpi(\Deltatilde(\RR))\) is nonempty if and only if \(\Deltatilde(\RR)\) is, this proves~\eqref{item:connected-image-complement-2}.
\end{proof}

%%%%%%%%%%%%%%%%%%%%%%%%%%%%%%%%%%%%%%%%%%%%%%%%%%%%%%%%%%%%%%%%%%%%%%%%%%%%%%%%%%%%%%%%%%%%
\subsection{Degree \(4\) conic bundles over \(\mathbb{R}\)}\label{sec:RealDegree4}
%%%%%%%%%%%%%%%%%%%%%%%%%%%%%%%%%%%%%%%%%%%%%%%%%%%%%%%%%%%%%%%%%%%%%%%%%%%%%%%%%%%%%%%%%%%%
For degree \(4\) conic bundles of the form \(Y_{\Deltatilde/\Delta}\), we give a topological interpretation of the triviality of the Brauer class \(\alphatildem{1}\) defined in Section~\ref{sec:alpha} (Proposition~\ref{prop:alphatilde-over-R}) and 
classify the cases from Theorem~\ref{thm:real} when the IJT obstruction or the real topological obstruction individually characterize rationality (Proposition~\ref{prop:real-equiv-rationality-criteria}), thereby extending results from~\cite{JJ23}.

The arguments in this section often rely on
Zeuthen's classification of the rigid isotopy classes of smooth quartic plane curves over \(\RR\), which states that \(\Delta(\RR)\) is one of the following: \(\emptyset\), one oval, two nested ovals, two non-nested ovals, three ovals, or four ovals \cite{Zeuthen1874,Klein1876}.

\begin{proposition}\label{prop:alphatilde-over-R}
    Let \(\varpi\colon \Deltatilde\to \Delta\) be a geometrically integral \'etale double cover of a smooth plane quartic over \(\RR\), and let \(Y_{\Deltatilde/\Delta}\) be the double cover defined in~\eqref{eq:DoubleCover}.
    Assume that \(\Ptildem{1}(\RR)\neq\emptyset\). Then the class \(\alphatildem{1}\in \Br \RR\) is  trivial if and only if the image of \(Y(\RR)\to\PP^2(\RR)\) contains the outside of \(\Delta(\RR)\).
\end{proposition}

When \(\Delta(\RR)\) is not one oval or two nested ovals, then rationality of \(Y\) implies that the image of \(Y(\RR)\to\PP^2(\RR)\) contains the outside of \(\Delta(\RR)\) by \cite{JJ23}, so Proposition~\ref{prop:alphatilde-over-R} implies \(\alphatildem{1}=0\). However, when \(\Delta(\RR)\) is two nested ovals, there are examples where \(Y\) is rational and \(\alphatildem{1}\neq 0\), as mentioned in Remark~\ref{rem:Real}\ref{it:rem-Real-alphatilde1-nonzero}. In \cite{JJ23}*{Example 4.10(1)} the image of \(Y(\RR)\) is the closed annulus enclosed by the two ovals, and in \cite{JJ23}*{Example 4.10(2)} the image of \(Y(\RR)\) is the union of the annulus and the disc enclosed by the interior oval.
\begin{proof}Let \(x \in \Ptildem{1}(\RR)\); write \(D_x\in \bPic^4_{\Deltatilde}\) for the associated effective divisor (see~\eqref{def:Ptilde1}) and \(\ell(x)\) for the line in \(\PP^2\) spanned by \(\piDelta_* D_x\).  The key observation is that \(\ell(x)(\RR) \smallsetminus (\ell(x)\cap \Delta)(\RR)\) meets the outside of \(\Delta(\RR)\).  If \(\alphatildem{1} = 0\), then \(Y_{\ell(x)}(\RR) \to \ell(x)(\RR)\) is surjective by Theorem~\ref{thm:alphatilde1}, and hence the image of \(Y(\RR)\to\PP^2(\RR)\) meets (and consequently contains) the outside of \(\Delta(\RR)\).
    Conversely, if the image of \(Y(\RR)\to\PP^2(\RR)\) contains the outside of \(\Delta(\RR)\), then there exists \(w \in \ell(x)(\RR)\smallsetminus (\ell(x)\cap \Delta)(\RR)\) such that \(Y_w(\RR) \neq \emptyset\).  Thus \(\alphatildem{1} = 0\) by Theorem~\ref{thm:alphatilde1}.
\end{proof}

\begin{prop}\label{prop:real-equiv-rationality-criteria}
Let \(\varpi\colon \Deltatilde\to \Delta\) be a geometrically integral \'etale double cover of a smooth plane quartic over \(\RR\), and let \(Y_{\Deltatilde/\Delta}\) be the double cover defined in~\eqref{eq:DoubleCover}. Assume that \(\Delta(\RR)\) is not one oval.

If \(\Delta(\RR)\) is two nested ovals, three ovals, or four ovals, then the four conditions of Theorem~\ref{thm:real} are also equivalent to
    \begin{enumerate}\setcounter{enumi}{4}
        \item\label{item:2nestedovals-connected} \(Y(\RR)\) is connected.
    \end{enumerate}
    
    If \(\Delta(\RR)\) is two nested ovals or four ovals then~\eqref{item:2nestedovals-connected} is additionally equivalent to
        \begin{enumerate}\setcounter{enumi}{5}
            \item\label{item:2nestedovals-IJT} The IJT obstruction vanishes for \(Y\).
    \end{enumerate}
    Moreover, the assumptions above are the strongest possible, i.e., if \(\Delta(\RR)\) is empty or two non-nested ovals then connectedness is not equivalent to rationality, and if \(\Delta(\RR)\) is empty, two non-nested ovals, or three ovals then vanishing of the IJT obstruction is not equivalent to rationality.
\end{prop}

\begin{remark}
    The part of the proposition that is new is the case of two nested ovals; the rest follows from prior work.  If \(\Delta(\RR)\) is four ovals, then \(Y\) is always rational~\cite{JJ23}*{Theorem 1.2(4)} so the theorem vacuously holds. If \(\Delta(\RR)\) is  three ovals, then the result that connectedness of \(Y(\RR)\) characterizes rationality is~\cite{JJ23}*{Theorem 1.2(3)}. That connectedness fails to characterize rationality if \(\Delta(\RR)\) is empty or two non-nested ovals is \cite{JJ23}*{Examples 4.2 and 4.4}. That the IJT obstruction fails to characterize rationality if \(\Delta(\RR)\) is two non-nested ovals is shown by~\cite{FJSVV}*{Theorem 1.4(i)}; if \(\Delta(\RR)\) is empty or three ovals, this is shown by~\cite{JJ23}*{Examples 4.1 and 4.6}.
\end{remark}
\begin{proof}
Given the remark above, we restrict ourselves to the case that \(\Delta(\RR)\) is two nested ovals.
In this case, \eqref{item:2nestedovals-IJT}\(\Rightarrow\)\eqref{item:2nestedovals-connected} by \cite{JJ23}*{Corollary 3.9}. This shows that~\eqref{item:2nestedovals-IJT} is equivalent to~\eqref{it:IJTandConnected}.

    It remains to show that \eqref{item:2nestedovals-connected}\(\Rightarrow\)\eqref{item:2nestedovals-IJT}. Assume for the sake of contradiction that~\eqref{item:2nestedovals-connected} holds and that \(Y\) has an IJT obstruction to rationality over \(\RR\). Then \(\piQuadricSurface\) does not have a section by Theorem~\ref{thm:real}, so \(\piQuadricSurface\) is not surjective on \(\RR\)-points by \cite{Witt-quadratic-forms}*{Satz 22}.  Thus there exists a point \(t \in \PP^1(\RR)\) such that \(Y_t(\RR)=\emptyset\), and so is defined by a definite quadratic form. In particular, there exists \(z\in\Gamma_t(\RR)\) such that \([\calF(Y/\PP^1)_z]\) is equal to the unique nontrivial element \(\mathbb{H}\in \Br \RR\). By Proposition~\ref{prop:construction-intersection-of-quadrics} and Corollary~\ref{cor:intro-inline}, 
    there exists an irrational complete intersection \(Z\subset\PP^5\) of two quadrics, containing a conic \(C\), such that \(\psi\colon\Bl_C Z\to\PP^2\) is a conic bundle with
    \[
    [(\Bl_C Z)_{\eta}] = [Y_\eta] + \mathbb H \in \Br \kk(\PP^2),\]
    so the image \(\psi((\Bl_C Z)(\RR))\) is equal to \((\PP^2(\RR) \smallsetminus \psi(Y(\RR))) \cup \Delta(\RR)\).
    Since \(Y(\RR)\) and \((\Bl_C Z)(\RR)\) are both 3-manifolds, the images \(\piX(Y(\RR))\) and \(\psi((\Bl_C Z)(\RR))\) do not have any connected components that are 1-manifolds.
    Thus, since \(Y(\RR)\) is connected and \(\Delta(\RR)\) is two nested ovals, the image \(\piX(Y(\RR))\) must be the region (diffeomorphic to a closed annulus) enclosed by the two ovals of \(\Delta(\RR)\). This implies \(Z(\RR)\) is disconnected, so Krasnov's real isotopy classification of 3-dimensional intersections of two quadrics \cite{krasnov-biquadrics}*{Theorem 5.4} (see also \cite{HT-intersection-quadrics}*{Sections 11.2 and 11.4}) implies that \(Z(\RR)\) is diffeomorphic to \(\mathbb S^3 \sqcup \mathbb S^3\) and that the Prym curve \(\Gamma\) has 6 real Weierstrass points. On the other hand, since we've assumed that \(\piQuadricSurface(Y(\RR))\subsetneq\PP^1(\RR)\) is connected, the fact that \(\Gamma\) has 6 real Weierstrass points implies that \(\piQuadricSurface\colon Y\to\PP^1\) must have a smooth fiber with signature \(({+,+,-,-})\). After a coordinate change on \(\PP^1\) we may assume that the fiber of \(\piQuadricSurface\) over \([1:0]\) has signature \(({+,+,-,-})\) \cite{FJSVV}*{Theorem 4.5}, i.e., that the quadratic form \(Q_1\) has signature \(({+,+,-})\). Then \((Q_1\geq 0) = \PP^2(\RR)\smallsetminus\mathbb D^2\subset\PP^2(\RR)\), and this region is contained in \(\piX(Y(\RR))\) by \cite{JJ23}*{Lemma 2.8}. Since \(\PP^2(\RR)\smallsetminus\mathbb D^2\) is non-orientable, as it is a M\"obius band, this contradicts our earlier conclusion that \(\piX(Y(\RR))\) is a closed annulus, so we have shown \eqref{item:2nestedovals-connected} implies~\eqref{item:2nestedovals-IJT}.
\end{proof}

%%%%%%%%%%%%%%%%%%%%%%%%%%%%%%%%%%%%%%%%%%%%%%%%%%%%%%%%%%%%%%%%%%%%%%%
%%%%%% Bibliography
%%%%%%%%%%%%%%%%%%%%%%%%%%%%%%%%%%%%%%%%%%%%%%%%%%%%%%%%%%%%%%%%%%%%%%%

\bibliographystyle{alpha}
\bibliography{references}

@article {beauville-ij,
    AUTHOR = {Beauville, Arnaud},
     TITLE = {Vari\'{e}t\'{e}s de {P}rym et jacobiennes interm\'{e}diaires},
   JOURNAL = {Ann. Sci. \'{E}c. Norm. Sup\'{e}r (4)},
  FJOURNAL = {Annales Scientifiques de l'\'{E}cole Normale Sup\'{e}rieure. Quatri\`eme
              S\'{e}rie},
    VOLUME = {10},
      YEAR = {1977},
    NUMBER = {3},
     PAGES = {309--391},
      ISSN = {0012-9593},
   MRCLASS = {14K30 (14C30 14J99 14M20)},
  MRNUMBER = {472843},
MRREVIEWER = {T. Oda},
       URL = {http://www.numdam.org/item?id=ASENS_1977_4_10_3_309_0},
}

@book {GS,
    AUTHOR = {Gille, Philippe and Szamuely, Tam\'{a}s},
     TITLE = {Central simple algebras and {G}alois cohomology},
    SERIES = {Cambridge Studies in Advanced Mathematics},
    VOLUME = {165},
   EDITION = {Second},
 PUBLISHER = {Cambridge University Press, Cambridge},
      YEAR = {2017},
     PAGES = {xi+417},
      ISBN = {978-1-316-60988-0; 978-1-107-15637-1},
   MRCLASS = {16K20 (14C35 14F22 19C30)},
  MRNUMBER = {3727161},
}

@incollection {IOOV,
    AUTHOR = {Ingalls, Colin and Obus, Andrew and Ozman, Ekin and Viray,
              Bianca},
     TITLE = {Unramified {B}rauer classes on cyclic covers of the projective
              plane},
 BOOKTITLE = {Brauer groups and obstruction problems},
    SERIES = {Progr. Math.},
    VOLUME = {320},
     PAGES = {115--153},
      NOTE = {With an appendix by Hugh Thomas},
 PUBLISHER = {Birkh\"{a}user/Springer, Cham},
      YEAR = {2017},
      ISBN = {978-3-319-46851-8; 978-3-319-46852-5},
   MRCLASS = {14F22 (14J28 14J50 15A66 16K50)},
  MRNUMBER = {3616009},
MRREVIEWER = {Giancarlo\ Lucchini Arteche},
       DOI = {10.1007/978-3-319-46852-5\_7},
       URL = {https://doi.org/10.1007/978-3-319-46852-5_7},
}

@article {DelfsKnebusch,
    AUTHOR = {Delfs, Hans and Knebusch, Manfred},
     TITLE = {Semialgebraic topology over a real closed field. {II}. {B}asic
              theory of semialgebraic spaces},
   JOURNAL = {Math. Z.},
  FJOURNAL = {Mathematische Zeitschrift},
    VOLUME = {178},
      YEAR = {1981},
    NUMBER = {2},
     PAGES = {175--213},
      ISSN = {0025-5874,1432-1823},
   MRCLASS = {14G30 (14A20)},
  MRNUMBER = {631628},
MRREVIEWER = {Thomas\ C.\ Craven},
       DOI = {10.1007/BF01262039},
       URL = {https://doi.org/10.1007/BF01262039},
}

@article {bruin,
    AUTHOR = {Bruin, Nils},
     TITLE = {The arithmetic of {P}rym varieties in genus 3},
   JOURNAL = {Compos. Math.},
  FJOURNAL = {Compositio Mathematica},
    VOLUME = {144},
      YEAR = {2008},
    NUMBER = {2},
     PAGES = {317--338},
      ISSN = {0010-437X},
   MRCLASS = {11G30 (14H40)},
  MRNUMBER = {2406115},
MRREVIEWER = {Jordi Gu\`ardia},
       DOI = {10.1112/S0010437X07003314},
}

@ARTICLE{bw-ij,
    AUTHOR = {Benoist, Olivier and Wittenberg, Olivier},
     TITLE = {Intermediate {J}acobians and rationality over arbitrary
              fields},
   JOURNAL = {Ann. Sci. \'{E}c. Norm. Sup\'{e}r. (4)},
  FJOURNAL = {Annales Scientifiques de l'\'{E}cole Normale Sup\'{e}rieure.
              Quatri\`eme S\'{e}rie},
    VOLUME = {56},
      YEAR = {2023},
    NUMBER = {4},
     PAGES = {1029--1086},
      ISSN = {0012-9593,1873-2151},
   MRCLASS = {14E08 (14G05 14H40 19A49)},
  MRNUMBER = {4650157},
       DOI = {10.24033/asens.2549},
       URL = {https://doi.org/10.24033/asens.2549},
}

@ARTICLE{FJSVV,
       author = {{Frei}, Sarah and {Ji}, Lena and {Sankar}, Soumya and {Viray}, Bianca and {Vogt}, Isabel},
        title = "{Curve classes on conic bundle threefolds and applications to rationality}",
   JOURNAL = {Algebr. Geom.},
  FJOURNAL = {Algebraic Geometry},
    VOLUME = {11},
      YEAR = {2024},
    NUMBER = {3},
     PAGES = {421--459},
       DOI = {10.14231/AG-2024-014},
}

@article {Witt-quadratic-forms,
    AUTHOR = {Witt, Ernst},
     TITLE = {Theorie der quadratischen {F}ormen in beliebigen {K}\"{o}rpern},
   JOURNAL = {J. Reine Angew. Math.},
  FJOURNAL = {Journal f\"{u}r die Reine und Angewandte Mathematik. [Crelle's
              Journal]},
    VOLUME = {176},
      YEAR = {1937},
     PAGES = {31--44},
      ISSN = {0075-4102},
   MRCLASS = {DML},
  MRNUMBER = {1581519},
       DOI = {10.1515/crll.1937.176.31},
}

@book {Lam-quadratic-forms,
    AUTHOR = {Lam, T. Y.},
     TITLE = {Introduction to quadratic forms over fields},
    SERIES = {Graduate Studies in Mathematics},
    VOLUME = {67},
 PUBLISHER = {American Mathematical Society, Providence, RI},
      YEAR = {2005},
     PAGES = {xxii+550},
      ISBN = {0-8218-1095-2},
   MRCLASS = {11Exx},
  MRNUMBER = {2104929},
       DOI = {10.1090/gsm/067},
}

@article {HT-intersection-quadrics,
    AUTHOR = {Hassett, Brendan and Tschinkel, Yuri},
     TITLE = {Rationality of complete intersections of two quadrics over
              nonclosed fields},
      NOTE = {With an appendix by Jean-Louis Colliot-Th\'{e}l\`ene},
   JOURNAL = {Enseign. Math.},
  FJOURNAL = {L'Enseignement Math\'{e}matique},
    VOLUME = {67},
      YEAR = {2021},
    NUMBER = {1-2},
     PAGES = {1--44},
      ISSN = {0013-8584},
   MRCLASS = {14E08 (14M10)},
  MRNUMBER = {4323991},
       DOI = {10.4171/lem/1001},
}

@article {JJ23,
    AUTHOR = {Ji, Lena and Ji, Mattie},
     TITLE = {Rationality of Real Conic Bundles With Quartic
              Discriminant Curve},
   JOURNAL = {Int. Math. Res. Not. IMRN},
  FJOURNAL = {International Mathematics Research Notices. IMRN},
      YEAR = {2024},
    NUMBER = {1},
     PAGES = {115--151},
      ISSN = {1073-7928,1687-0247},
   MRCLASS = {14},
  MRNUMBER = {4686647},
       DOI = {10.1093/imrn/rnad003},
       URL = {https://doi.org/10.1093/imrn/rnad003},
}

@article {krasnov-biquadrics,
    AUTHOR = {Krasnov, V. A.},
     TITLE = {On the intersection of two real quadrics},
   JOURNAL = {Izv. Ross. Akad. Nauk Ser. Mat.},
  FJOURNAL = {Izvestiya Rossiiskoi Akademii Nauk. Seriya Matematicheskaya},
    VOLUME = {82},
      YEAR = {2018},
    NUMBER = {1},
     PAGES = {97--150},
      ISSN = {1607-0046},
   MRCLASS = {14P25 (14M10 14N25 57R19)},
  MRNUMBER = {3749598},
       DOI = {10.4213/im8535},
}

@article {Amitsur,
    AUTHOR = {Amitsur, S. A.},
     TITLE = {Generic splitting fields of central simple algebras},
   JOURNAL = {Ann. of Math. (2)},
  FJOURNAL = {Annals of Mathematics. Second Series},
    VOLUME = {62},
      YEAR = {1955},
     PAGES = {8--43},
      ISSN = {0003-486X},
   MRCLASS = {09.3X},
  MRNUMBER = {70624},
       DOI = {10.2307/2007098},
       URL = {https://doi.org/10.2307/2007098},
}

@book {Mangolte-real,
    AUTHOR = {Mangolte, Fr\'{e}d\'{e}ric},
     TITLE = {Real algebraic varieties},
    SERIES = {Springer Monographs in Mathematics},
      NOTE = {Translated from the 2017 French original [3727103] by Catriona
              Maclean},
 PUBLISHER = {Springer, Cham},
      YEAR = {2020},
     PAGES = {xviii+444},
      ISBN = {978-3-030-43104-4; 978-3-030-43103-7},
   MRCLASS = {14Pxx},
  MRNUMBER = {4179588},
       DOI = {10.1007/978-3-030-43104-4},
       URL = {https://doi.org/10.1007/978-3-030-43104-4},
}

@book {CTS-Brauer-book,
    AUTHOR = {Colliot-Th\'{e}l\`ene, Jean-Louis and Skorobogatov, Alexei N.},
     TITLE = {The {B}rauer-{G}rothendieck group},
    SERIES = {Ergebnisse der Mathematik und ihrer Grenzgebiete. 3. Folge. A
              Series of Modern Surveys in Mathematics [Results in
              Mathematics and Related Areas. 3rd Series. A Series of Modern
              Surveys in Mathematics]},
    VOLUME = {71},
 PUBLISHER = {Springer, Cham},
      YEAR = {2021},
     PAGES = {xv+453},
      ISBN = {978-3-030-74247-8; 978-3-030-74248-5},
   MRCLASS = {14F22 (14E08 14G05 14G12 14K05)},
  MRNUMBER = {4304038},
}

@article {Springer52,
    AUTHOR = {Springer, Tonny Albert},
     TITLE = {Sur les formes quadratiques d'indice z\'{e}ro},
   JOURNAL = {C. R. Acad. Sci. Paris},
  FJOURNAL = {Comptes Rendus Hebdomadaires des S\'{e}ances de l'Acad\'{e}mie
              des Sciences},
    VOLUME = {234},
      YEAR = {1952},
     PAGES = {1517--1519},
      ISSN = {0001-4036},
   MRCLASS = {09.1X},
  MRNUMBER = {47021},
}

@misc{Reid-thesis,
    author={Reid, Miles},
     title={The complete intersection of two or more quadrics},
     year={1972},
     NOTE = {Ph.D thesis (Trinity College, Cambridge)},
      HOWPUBLISHED = {\url{http://homepages.warwick.ac.uk/~masda/3folds/qu.pdf}},
}

@ARTICLE{KP-Fano-3folds-rank1,
    AUTHOR = {Kuznetsov, Alexander and Prokhorov, Yuri},
     TITLE = {Rationality of {F}ano threefolds over non-closed fields},
   JOURNAL = {Amer. J. Math.},
  FJOURNAL = {American Journal of Mathematics},
    VOLUME = {145},
      YEAR = {2023},
    NUMBER = {2},
     PAGES = {335--411},
      ISSN = {0002-9327},
   MRCLASS = {14},
  MRNUMBER = {4570985},
}

@article {HT-cycle,
    AUTHOR = {Hassett, Brendan and Tschinkel, Yuri},
     TITLE = {Cycle class maps and birational invariants},
   JOURNAL = {Comm. Pure Appl. Math.},
  FJOURNAL = {Communications on Pure and Applied Mathematics},
    VOLUME = {74},
      YEAR = {2021},
    NUMBER = {12},
     PAGES = {2675--2698},
      ISSN = {0010-3640},
   MRCLASS = {14E08 (14M20)},
  MRNUMBER = {4373165},
       DOI = {10.1002/cpa.21967},
}

@article {Zeuthen1874,
    AUTHOR = {Zeuthen, Hieronymus Georg},
     TITLE = {Sur les diff\'{e}rentes formes des courbes planes du quatri\`{e}me ordre},
   JOURNAL = {Math. Ann.},
  FJOURNAL = {Mathematische Annalen},
    VOLUME = {7},
      YEAR = {1874},
     PAGES = {410--432},
}

@article {Lichtenbaum,
    AUTHOR = {Lichtenbaum, Stephen},
     TITLE = {Duality theorems for curves over {$p$}-adic fields},
   JOURNAL = {Invent. Math.},
  FJOURNAL = {Inventiones Mathematicae},
    VOLUME = {7},
      YEAR = {1969},
     PAGES = {120--136},
      ISSN = {0020-9910,1432-1297},
   MRCLASS = {14.40},
  MRNUMBER = {242831},
MRREVIEWER = {Manfred\ Herrmann},
       DOI = {10.1007/BF01389795},
       URL = {https://doi.org/10.1007/BF01389795},
}

@article {poonen-stoll,
    AUTHOR = {Poonen, Bjorn and Stoll, Michael},
     TITLE = {The {C}assels-{T}ate pairing on polarized abelian varieties},
   JOURNAL = {Ann. of Math. (2)},
  FJOURNAL = {Annals of Mathematics. Second Series},
    VOLUME = {150},
      YEAR = {1999},
    NUMBER = {3},
     PAGES = {1109--1149},
      ISSN = {0003-486X,1939-8980},
   MRCLASS = {11G10 (11G30 14H40 14K15)},
  MRNUMBER = {1740984},
MRREVIEWER = {Tam\'{a}s\ Szamuely},
       DOI = {10.2307/121064},
       URL = {https://doi.org/10.2307/121064},
}

@incollection {mumford-prym,
    AUTHOR = {Mumford, David},
     TITLE = {Prym varieties. {I}},
 BOOKTITLE = {Contributions to analysis (a collection of papers dedicated to
              {L}ipman {B}ers)},
     PAGES = {325--350},
      YEAR = {1974},
   MRCLASS = {14H40 (14K25 32G20)},
  MRNUMBER = {0379510},
MRREVIEWER = {H. H. Martens},
}

@article {Magma,
    AUTHOR = {Bosma, Wieb and Cannon, John and Playoust, Catherine},
     TITLE = {The {M}agma algebra system. {I}. {T}he user language},
      NOTE = {Computational algebra and number theory (London, 1993)},
   JOURNAL = {J. Symbolic Comput.},
  FJOURNAL = {Journal of Symbolic Computation},
    VOLUME = {24},
      YEAR = {1997},
    NUMBER = {3-4},
     PAGES = {235--265},
      ISSN = {0747-7171},
   MRCLASS = {68Q40},
  MRNUMBER = {MR1484478},
       DOI = {10.1006/jsco.1996.0125},
       URL = {http://dx.doi.org/10.1006/jsco.1996.0125},
}

@article {clemensgriffiths,
    AUTHOR = {Clemens, C. Herbert and Griffiths, Phillip A.},
     TITLE = {The {I}ntermediate {J}acobian of the Cubic Threefold},
   JOURNAL = {Ann. of Math.},
  FJOURNAL = {Annals of Mathematics},
    VOLUME = {95},
      YEAR = {1972},
    NUMBER = {2},
     PAGES = {281--356},
   MRCLASS = {14J10 (14G13 14J05 14K20 14M20 14N99)},
  MRNUMBER = {0302652 },
MRREVIEWER = {H. Popp},
       DOI = {10.2307/1970801},
}

@incollection {IskovskikhProkhorov,
    AUTHOR = {Iskovskikh, V. A. and Prokhorov, Yu. G.},
     TITLE = {Fano varieties},
 BOOKTITLE = {Algebraic geometry, {V}},
    SERIES = {Encyclopaedia Math. Sci.},
    VOLUME = {47},
     PAGES = {1--247},
 PUBLISHER = {Springer, Berlin},
      YEAR = {1999},
      ISBN = {3-540-61468-0},
   MRCLASS = {14J45 (14E07 14F22 14K30)},
  MRNUMBER = {1668579},
}

@article {bw-cg,
    AUTHOR = {Benoist, Olivier and Wittenberg, Olivier},
     TITLE = {The {C}lemens-{G}riffiths method over non-closed fields},
   JOURNAL = {Algebr. Geom.},
  FJOURNAL = {Algebraic Geometry},
    VOLUME = {7},
      YEAR = {2020},
    NUMBER = {6},
     PAGES = {696--721},
      ISSN = {2313-1691},
   MRCLASS = {14E08 (14G27 14K30)},
  MRNUMBER = {4156423},
MRREVIEWER = {Sajad Salami},
       DOI = {10.14231/ag-2020-025},
}

@article {BhargavaGrossWang,
    AUTHOR = {Bhargava, Manjul and Gross, Benedict H. and Wang, Xiaoheng},
     TITLE = {A positive proportion of locally soluble hyperelliptic curves
              over {$\Bbb Q$} have no point over any odd degree extension},
      NOTE = {With an appendix by Tim Dokchitser and Vladimir Dokchitser},
   JOURNAL = {J. Amer. Math. Soc.},
  FJOURNAL = {Journal of the American Mathematical Society},
    VOLUME = {30},
      YEAR = {2017},
    NUMBER = {2},
     PAGES = {451--493},
      ISSN = {0894-0347,1088-6834},
   MRCLASS = {11G30 (14G05 14H50)},
  MRNUMBER = {3600041},
MRREVIEWER = {James\ H.\ Stankewicz},
       DOI = {10.1090/jams/863},
       URL = {https://doi.org/10.1090/jams/863},
}

@article {ProkhorovSurvey,
    AUTHOR = {Prokhorov, Yu. G.},
     TITLE = {The rationality problem for conic bundles},
   JOURNAL = {Uspekhi Mat. Nauk},
  FJOURNAL = {Uspekhi Matematicheskikh Nauk},
    VOLUME = {73},
      YEAR = {2018},
    NUMBER = {3(441)},
     PAGES = {3--88},
      ISSN = {0042-1316,2305-2872},
   MRCLASS = {14E08 (14E07 14E30 14M20)},
  MRNUMBER = {3807895},
MRREVIEWER = {Jaros\l aw\ A.\ Wi\'sniewski},
       DOI = {10.4213/rm9811},
       URL = {https://doi.org/10.4213/rm9811},
}

@article {Sarkisov,
    AUTHOR = {Sarkisov, V. G.},
     TITLE = {On conic bundle structures},
   JOURNAL = {Izv. Akad. Nauk SSSR Ser. Mat.},
  FJOURNAL = {Izvestiya Akademii Nauk SSSR. Seriya Matematicheskaya},
    VOLUME = {46},
      YEAR = {1982},
    NUMBER = {2},
     PAGES = {371--408, 432},
      ISSN = {0373-2436},
   MRCLASS = {14J40 (14E05 14J10)},
  MRNUMBER = {651652},
MRREVIEWER = {Daniel\ Coray},
}

@article {Klein1876,
    AUTHOR = {Klein, Felix},
     TITLE = {\"{U}ber den {V}erlauf der {A}belschen {I}ntegrale bei den {K}urven vierten {G}rades},
   JOURNAL = {Math. Ann.},
  FJOURNAL = {Mathematische Annalen},
    VOLUME = {10},
      YEAR = {1876},
     PAGES = {365--397},
}

\end{document}